\documentclass[10pt, journal, twocolumn]{ieeetran}                                                       
\IEEEoverridecommandlockouts                              

\usepackage{pgfplots}
\usepackage{graphicx}      
\usepackage{amsmath,amsfonts} 
\usepackage{enumerate}
\usepackage{color}
\usepackage{amssymb}
\usepackage{mathrsfs}
\usepackage{tikz,pgfplots}
\usepackage{tabu}
\usepackage{mathtools}
\usepackage{epstopdf}

\usepackage{hyperref, balance}

\usepackage{subcaption}
\newtheorem{theorem}{Theorem}
\newtheorem{definition}{Definition}
\newtheorem{proposition}{Proposition}
\newtheorem{lemma}{Lemma}
\newtheorem{corollary}{Corollary}

\newtheorem{remark}{Remark}
\newtheorem{assumption}{Assumption}

\newcommand{\bs}{\boldsymbol}
\newcommand{\mc}{\mathcal}
\newcommand{\bb}{\mathbb}
\newcommand{\R}{\bb R}
\newcommand{\0}{\bs 0}
\newcommand{\1}{\bs 1}

\newcommand{\norm}[1]{\left\|#1\right\|}

\newcommand{\dom}{\operatorname{dom}}

\newcommand{\blue}{\textcolor{black}}
\newcommand{\argmin}{\operatorname{argmin}}

\newcommand{\fix}{\mathrm{fix}}
\newcommand{\proj}{\mathrm{proj}}
\newcommand{\prox}{\mathrm{prox}}
\newcommand{\Id}{\mathrm{Id}}

\newcommand{\diag}{\operatorname{diag}}
\newcommand{\col}{\operatorname{col}}
\newcommand{\zer}{\operatorname{zer}}

\newcommand{\nc}{\mathrm{N}}
\newcommand{\avg}{\mathrm{avg}}

\title{
Semi-decentralized generalized Nash equilibrium seeking in monotone aggregative games
}

\author{Giuseppe Belgioioso and Sergio Grammatico 
\thanks{G. Belgioioso is with the Automatic Control Laboratory, Swiss Federal Institute of Technology (ETH) Z\"{u}rich, Switzerland. S. Grammatico is with the Delft Center for Systems and Control (DCSC), TU Delft, The Netherlands. E-mail addresses: \texttt{gbelgioioso@ethz.ch}, \texttt{s.grammatico@tudelft.nl}. This work was partially supported by NWO under research projects OMEGA (613.001.702) and P2P-TALES (647.003.003), and by the ERC under research project COSMOS (802348).

}
}

\begin{document}

\maketitle
\thispagestyle{empty}
\pagestyle{empty}

\begin{abstract}      
We address the generalized Nash equilibrium seeking problem for a population of agents playing aggregative games with affine coupling constraints. We focus on semi-decentralized communication architectures, where there is a central coordinator able to gather and broadcast signals of aggregative nature to the agents.
By exploiting the framework of monotone operator theory and operator splitting, we first critically review the most relevant available algorithms and then design two novel schemes: (i) a single-layer, fixed-step algorithm with convergence guarantee for general (non-cocoercive, non-strictly) monotone aggregative games and (ii) a single-layer proximal-type algorithm for a class of monotone aggregative games with linearly coupled cost functions. We also design novel accelerated variants of the algorithms via (alternating) inertial and over-relaxation steps. Finally, we show via numerical simulations that the proposed algorithms outperform those in the literature in terms of convergence speed.
\end{abstract}

\section{Introduction} \label{sec:Intro}

\subsection{Aggregative games}
\IEEEPARstart{A}{n} aggregative game is a set of coupled optimization problems, each associated with an autonomous agent, i.e., an independent decision maker, where the cost function of each agent depends on some aggregate effect of all the agents in the population \cite{kukushkin:04}, \cite{jensen:10}, \cite{cornes:12}. Namely, the aggregative feature arises whenever each agent is affected by the overall population behavior, hence not by some specific agents. In general, such a special feature is typical of incentive-based control in competitive markets \cite{ma:hu:spanos:14} and in fact  engineering applications of aggregative games span from demand side management in the smart grid \cite{chen:li:louie:vucetic:14} \cite{ye:hu:17} and charging control for plug-in electric vehicles \cite{ma2016efficient}, \cite{ma:zou:ran:shi:hiskens:16}, \cite{liu:18}, to spectrum sharing in wireless networks \cite{zhou:17} and network congestion control \cite{barrera:garcia:15}. With these motivating applications in mind, aggregative games have been receiving high research interest, within the operations research \cite{koshal:nedic:shanbhag:16} and especially the automatic control \cite{grammatico:17}, \cite{belgioioso:grammatico:17cdc}, \cite{liang:yi:hong:17}, \cite{paccagnan2018nash}, \cite{deng:liang:19}, \cite{depersis:grammatico:20} communities. Researchers have in fact studied and proposed solutions to the generalized Nash equilibrium problem (GNEP) in aggregative games, which is the problem to compute a set of decisions such that each is individually optimal given the others. Remarkably, the aggregative structure has been exploited to mitigate the computational complexity for large population size, and in fact the proposed solution algorithms are primarily non-centralized, i.e., almost (semi-) decentralized and distributed, where the computations by the agents are fully decoupled. Essentially, in semi-decentralized algorithms, the agents do not  communicate with each other, but rely on a reliable central coordinator (e.g. an aggregator) that gathers the local decisions in aggregative form and then broadcasts (incentive) signals, e.g. dual variables, to all the agents \cite{belgioioso:grammatico:17cdc}. On the other hand, in distributed algorithms, there is no central coordinator, so the agents communicate with each other to cooperatively estimate or reach consensus on the signals of common interest, e.g. dual and auxiliary variables. The latter algorithmic setup is also called partial-decision information \cite{gadjov2020single}, \cite{belgioioso2020distributedN}, because the agents do not have direct access to the aggregative effect on their cost functions, thus they should estimate it via reliable, truthful peer-to-peer communications, e.g. via cooperative consensus protocols. These features motivate us to focus on the semi-decentralized algorithmic structure in this paper.

\subsection{Literature review}

The literature on semi-decentralized GNE seeking in aggregative games is quite recent. In \cite{belgioioso:grammatico:17cdc} Belgioioso and Grammatico designed the first semi-decentralized GNE seeking algorithm for (non-strictly/strongly, non-cocoercive) \textit{monotone} aggregative games\footnote{For ease of reading, with (strict/strongly) monotone game, we mean game with (strict/strongly) monotone pseudo-subdifferential mapping (\S \ref{subsec:VA}).}, where the algorithm derivation relies on the so-called forward-backward-forward (FBF) operator splitting. In parallel, for the class of \textit{strongly monotone} games, Yi and Pavel proposed the first preconditioned forward-backward (pFB) operator splitting method \cite{yi2017distributed}, \cite{yi2019operator}, which is applicable to aggregative games with semi-decentralized algorithmic structure - as shown in \cite{belgioioso2018projected}, the outcome of the pFB method for aggregative games is in fact the so-called asymmetric project algorithm (APA) \cite[\S 12.5.1]{facchinei:pang}, proposed for aggregative games in \cite{paccagnan2018nash}. Effectively, \cite{belgioioso:grammatico:17cdc} and \cite{yi2017distributed} are the first works to adopt an elegant and general mathematical approach based on monotone operator theory \cite{bauschke2017convex} to explicitly model GNEPs, to decouple the coupling constraints via Lagrangian duality, and in turn to exploit operator splitting methods for systematically designing (non-centralized) GNE seeking algorithms. Next, we refer to some other relevant GNE seeking algorithms for or applicable to aggregative games. For a class of \textit{unconstrained} \textit{strictly monotone} games, in \cite{ye:hu:17}, Ye and Hu proposed continuous-time saddle-point dynamics. For \textit{strictly monotone} games with \textit{equality} coupling constraints, in \cite{liang:yi:hong:17}, Liang, Yi and Hong proposed continuous-time projected pseudo-gradient dynamics paired with discontinuous dynamics for dual and auxiliary variables. For \textit{unconstrained}, \textit{strongly monotone} aggregative games, in \cite{deng:liang:19}, Deng and Liang proposed continuous-time, integral consensus-based dynamics. Recently, in \cite{depersis:grammatico:20}, De Persis and Grammatico proposed continuous-time, integral dynamics for a class of \textit{strongly monotone} aggregative games.

From the literature on (semi-decentralized) GNE seeking in aggregative games, several critical issues emerge. First, the solution methods available for general (non-strictly, non-cocoercive) monotone aggregative games are limited to algorithms that require at least two demanding computations (projections) and two communications (between the agents and the coordinator) at each iteration, see e.g. the FBF \cite{belgioioso:grammatico:17cdc} and extra-gradient (EG) based methods \cite[\S 12.6.1]{facchinei:pang}; instead, computationally convenient algorithms, e.g. the pFB \cite{yi2019operator}, require strong monotonicity of the game.  Surprisingly, there is currently no single-communication-per-iteration, fixed-step algorithm for merely monotone aggregative games. For instance, the pFB method does not always converge in merely monotone games, not even under vanishing step sizes \cite{grammatico:18}. From a practical perspective, the available algorithms may require a large number of iterations, and in particular of communications between the agents and the central coordinator, to converge. For example, algorithms based on the iterative Tikhonov regularization (ITR) \cite{kannan2012distributed} require double-layer vanishing step sizes, which considerably slows down convergence. 
Finally, often, the local step sizes of the algorithms are based on global properties of the game data, that however should be unknown to the local agents in practice - on the contrary, little or no coordination among agents should be necessary to set the step sizes with guaranteed convergence.

\subsection{Contribution}

In this paper, we fully exploit monotone operator theory and operator splitting methodologies to study and address the main technical and computational issues that currently afflict (semi-decentralized) GNE seeking in aggregative games. Specifically, our main contributions are summarized next:

\begin{enumerate}[1.]
\item We characterize the available (semi-decentralized) algorithms with a general operator-theoretic perspective, which allows us to establish basically the most general convergence results and draw a fair technical comparison among these algorithms (\S \ref{sec:CREA}), as well as to possibly improve convergence speed, e.g. via inertial accelerations;

\item We present the first single-layer, single-communication-per-iteration, fixed-step algorithm for (non-strictly, non-cocoercive) monotone aggregative games (\S \ref{sec:FoRB}) - essentially, the most desirable algorithmic features for the most general class of monotone aggregative games one could hope for;

\item We present a very fast, single-layer, single-communication-per-iteration, fixed-step, proximal algorithm for a class of (non-strictly) monotone aggregative games with linearly coupled cost functions (\S \ref{sec:cPPP}) - essentially, the most desirable algorithmic features with the fastest convergence ever experienced for a special, popular, class of monotone aggregative games;

\item We design an alternating inertial acceleration scheme which is applicable to some algorithms (\S \ref{sec:AI}) and that, remarkably, in some particular cases outperforms the classic inertial acceleration in terms of numerical convergence - mathematically, we prove that our alternating inertia preserves averagedness of operators, thus the convergence is desirably Fej\'er monotone.
\end{enumerate}

\subsection{Notation and definitions}
\subsubsection*{Basic notation}
$\R$ denotes the set of real numbers, and $\overline{\R} := \R \cup \{\infty\}$ the set of extended real numbers. $\bs{0}$ ($\bs{1}$) denotes a matrix/vector with all elements equal to $0$ ($1$); to improve clarity, we may add the dimension of these matrices/vectors as subscript. $A \otimes B$ denotes the Kronecker product between the matrices $A$ and $B$. For a square matrix $A \in \R^{n \times n}$, it transpose is $A^\top$, $[A]_{i,j}$ represents the element on the row $i$ and column $j$. $A \succ 0$ ($\succeq 0$) stands for positive definite (semidefinite) matrix. Given $A \succ 0$, $\norm{\cdot}_A $ denotes the $A$-induced norm, such that $\norm{x}_{A}=x^\top A x$. $\left\| A \right\|$ denotes the largest singular value of $A$; \blue{$\text{eig}_{\max}(A)$ and $\text{eig}_{\min}(A)$ denote, respectively, the largest and the smallest eigenvalues of $A$.} Given $N$ scalars, $a_1,\ldots, a_N$, $\diag(a_1,\ldots, a_N)$ denotes the diagonal matrix with $a_1,\ldots, a_N$ on the main diagonal. Given $N$ vectors $x_1, \ldots, x_N \in \R^n$, $\boldsymbol{x} := \col\left(x_1,\ldots,x_N\right) = [ x_1^\top, \ldots , x_N^\top ]^\top$.

\smallskip
\subsubsection*{Operator-theoretic definitions}
$\Id(\cdot)$ denotes the identity operator. The mapping $\iota_{S}:\R^n \rightarrow \{ 0, \, \infty \}$ denotes the indicator function for the set $\mc{S} \subseteq \R^n$, i.e., $\iota_{\mc S}(x) = 0$ if $x \in \mc S$, $\infty$ otherwise. For a closed set $S \subseteq \R^n$, the mapping $\proj_{\mc S}:\R^n \rightarrow \mc S$ denotes the projection onto $\mc S$, i.e., $\proj_{\mc S}(x) = \argmin_{y \in \mc S} \left\| y - x\right\|$. The set-valued mapping $\nc_{S}: \R^n \rightrightarrows \R^n$ denotes the normal cone operator for the set $S \subseteq \R^n$, i.e., 
$\nc_{\mc S}(x) = \varnothing$ if $x \notin S$, $\left\{ v \in \R^n \mid \sup_{z \in \mc S} \, v^\top (z-x) \leq 0  \right\}$ otherwise.
For a function $\psi: \R^n \rightarrow \overline{\R}$, $\dom(\psi) := \{x \in \R^n \mid \psi(x) < \infty\}$; $\partial \psi: \dom(\psi) \rightrightarrows {\R}^n$ denotes its subdifferential set-valued mapping, defined as $\partial \psi(x) := \{ v \in \R^n \mid \psi(z) \geq \psi(x) + v^\top (z-x)  \textup{ for all } z \in {\rm dom}(\psi) \}$; $\prox_{\psi}(x) = \argmin_{y \in \R^n} \psi(y) + \frac{1}{2}\left\| y - x\right\|^2$ denotes its proximal operator. 
A set-valued mapping $\mathcal{F} : \R^n \rightrightarrows \R^n$ is
(strictly) monotone if $(u-v)^\top  (x-y) \geq (>) \, 0$ for all $x \neq y \in \R^n$, $u \in \mathcal{F} (x)$, $v \in \mathcal{F} (y)$; $\mathcal{F} $ is $\eta$-strongly monotone, with $\eta>0$, if 
$(u-v)^\top (x-y) \geq \eta \left\| x-y \right\|^2$ for all $x \neq y \in \R^n$, $u \in \mathcal{F} (x)$, $v \in \mathcal{F} (y)$.
${\rm J}_{\mathcal{F} }:=(\Id + \mathcal{F} )^{-1}$ denotes the resolvent operator of $\mathcal{F} $
; $\fix\left( \mathcal{F}\right) := \left\{ x \in \R^n \mid x \in \mathcal{F}(x) \right\}$ and $\zer\left( \mathcal{F}\right) := \left\{ x \in \R^n \mid 0 \in \mathcal{F}(x) \right\}$ denote the set of fixed points and of zeros, respectively.

\section{The generalized Nash equilibrium problem in aggregative games} \label{sec:PS}
\subsection{Problem statement}
We consider a set of $N$ agents, where each agent $i \in \mc I := \{1,\ldots, N \}$ shall choose its decision variable (i.e., strategy) $x_i$ from the local decision set $\Omega_i \subseteq \mathbb{R}^n$ with the aim of minimizing its local cost function $J_i\left( x_i, \bs{x}_{-i} \right)$, which depends on the local variable $x_i$ (first argument) and on the decision variables of the other agents, $\bs{x}_{-i} := \col\left( \{ x_j \}_{j\in \mc I\backslash \{ i \}} \right) \in \R^{n(N-1)}$ (second argument).

\smallskip
In this paper, we focus on the class of \textit{aggregative games}, where the cost function of each agent depends on the local decision variable and on the value of the aggregation, i.e.,
\begin{equation} \label{eq:sigma} \textstyle
\avg(\boldsymbol{x}) := 
\frac{1}{N}\sum_{i=1}^{N} x_i.
\end{equation}
Specifically, we consider local cost functions of the form
\begin{equation} \label{eq:CFi} \textstyle
J_i(x_i, \bs x_{-i}) := \textstyle  g_i(x_i) + f_i \left( x_i, \avg(\boldsymbol{x}) \right),
\end{equation}
where $g_i$ and $f_i$ satisfy the following assumptions.

\smallskip
\begin{assumption} \label{ass:CFn}
For each $i\in \mc I$, the function $g_i$ is continuous (possibly non-differentiable) and convex, and $f_i(\, \cdot \,,\frac{1}{N}\cdot + \, y)$ is continuously differentiable and convex, for any $y \in \R^n$.
{\hfill $\square$}
\end{assumption}

\smallskip
Cost functions as in \eqref{eq:CFi} are the most general considered in the literature of monotone games \cite[Rem.~1]{yi2018distributed}, \cite[\S~12]{palomar2010convex}.

\smallskip
Furthermore, we consider \textit{generalized games}, where the coupling among the agents arises not only via the cost functions, but also via their feasible decision sets. In our setup, the coupling constraints are described by an affine function, $\bs{x} \mapsto A \bs x - b$, where $A:= \left[ A_1| \ldots| A_N \right] \in \R^{m \times nN}$, $b := \sum_{i =1}^N b_i \in \R^m$. Thus, the global feasible set reads as
\begin{equation} 
\label{eq:calX}
\bs{\mc X}\:= \left( \prod_{i \in \mc I} \Omega_i \right) \bigcap \left\{ \bs x \in \R^{nN} | \, A \bs x - b \leq \boldsymbol{0}_m \right\} \subseteq \R^{nN};
\end{equation}
while the feasible decision set of each agent $i \in \mc I$ is characterized by the set-valued mapping $\mc X_i$, defined as
$$ \textstyle
\mc X_i(\bs x_{-i}) := \big\{ y_i \in \Omega_i | \, A_i y_i -b_i \leq \sum_{j \neq i}^N (b_j-A_j x_j) \big\},
$$
where $A_i \in \R^{m \times n}$ and $b_i$ are local parameters that define how agent $i$ is involved in the coupling constraints.

\smallskip
\begin{remark}[Affine constraints] Affine coupling constraints, as considered in \eqref{eq:calX}, are the most common in the literature of monotone games, see for example \cite{grammatico:17}, \cite{liang:yi:hong:17}, \cite{paccagnan2018nash}, \cite{yi2019operator}.
\blue{For the sake of compactness, we did not include coupling equality constraints in \eqref{eq:calX}. However, all the results in the remainder of the paper can be straightforwardly adapted to cover this case.}

{\hfill $\square$}
\end{remark}

\smallskip
Next, let us formalize standard convexity and closedness assumptions for the constraint sets.

\smallskip
\begin{assumption} \label{ass:CCF}
For each $i\in \mc I$, the local set $\Omega_i \subseteq \R^n$ is nonempty, closed and convex. Moreover, the global set $\bs{\mc X}$ satisfies Slater's constraint qualification.
{\hfill $\square$}
\end{assumption}
\smallskip

In summary, the aim of each agent $i$, given the aggregate decision $\avg(\bs x)$, is to choose a strategy, $x_i^*$, that solves its local convex optimization problem according to the game setup previously described, i.e.,  for all $i \in \mc I$
\begin{align}\label{eq:Game}
\left\{
\begin{array}{c l}
\underset{x_i \in \, \R^n}{\argmin}&
J_i \big( x_i, \bs x_{-i} \big) 
=
g_i(x_i) + f_i \left( x_i, \avg(\boldsymbol{x}) \right)
\\
\text{ s.t. }   &
 x_i \in \Omega_i\\
&   A_i x_i \leq b_i +\sum_{j \neq i}^N (b_j- A_j x_j)
\end{array} 
\right. 
\end{align}
where the last constraint is equivalent to $A \bs x - b \leq \0$.
From a game-theoretic perspective, we consider the problem to compute a Nash equilibrium \cite{facchinei2010generalized}, as formalized next.

\smallskip 
\begin{definition}[Generalized $\varepsilon-$Nash equilibrium]
A collective strategy $\bs x^*\in \bs{\mc X}$ is a generalized $\varepsilon-$Nash equilibrium ($\varepsilon-$GNE) of the game in \eqref{eq:Game} if, for all $i\in \mc I$:
\begin{align} \label{eq:eps_GNE}
J_i\left( x^*_{i}, \bs x^*_{-i} \right) \leq 
\inf\left\{ J_{i}(y, \, \bs x^*_{-i}) + \varepsilon \, \mid \, y \in \mc X_i(\bs x^*_{-i}) \right\}.
\end{align}
If \eqref{eq:eps_GNE} holds with $\varepsilon = 0$, then $\bs x^*$ is a GNE.
{\hfill $\square$}
\end{definition}

\smallskip
In other words, a set of strategies is a Nash equilibrium if no agent can improve its objective function by unilaterally changing its strategy to another feasible one.

\smallskip
\begin{remark}[Existence of a GNE]
\label{rem:ExGNE}
If Assumption \ref{ass:CCF} holds with bounded local strategy sets $\Omega_i$'s, the existence of a GNE follows from Brouwer's fixed-point theorem \cite[Prop.~12.7]{palomar2010convex}, while uniqueness does not hold in general.
{ \hfill $\square$}
\end{remark}

\subsection{Nash vs Aggregative (or Wardrop) Equilibria}

In aggregative games with cost functions as in \eqref{eq:CFi}, the condition in \eqref{eq:eps_GNE} specializes as: for all $i \in \mc I$ and $y \in \mc X_i(\bs x^*_{-i}) $
\begin{align*} 
g_i(x_i^*) + f_i\left( x^*_{i}, \avg(\bs x^*) \right) \leq 
 g_{i}(y) 
\textstyle
+f_i \big( y, \, \frac{1}{N} y +
\frac{1}{N} \sum_{j\neq i}^N x_j^* \big),
\end{align*}
where the decision variable of agent $i$, i.e., $x_i^*$, appears also in the second argument of $f_i$, since $x_i^*$ contributes to form the average strategy, i.e., $\avg(\bs x^*) = \frac{1}{N} x_i^* + \frac{1}{N} \sum_{j\neq i}^N x_j^* $.

The concept of \textit{aggregative (or Wardrop) equilibrium} (formalized in Definition \ref{def:GAE}) springs from the intuition that the contribution of each agent to the average strategy decreases as the population size grows. 
Technically, the influence of the decision variable of agent $i$ on the second argument of its cost function $f_i$ vanishes as $N$ grows unbounded.

\smallskip
\begin{definition}[Generalized Aggregative equilibrium]
\label{def:GAE}
A collective strategy $\bs x^\star  \in \bs{\mc X}$ is a generalized aggregative equilibrium (GAE) of the game in \eqref{eq:Game} if, for all $i\in \mc I$:
\begin{multline*} 
g_i(x_i^\star) + f_i\left( x^\star_{i}, \avg(\bs x^\star) \right) \leq \\
\inf \left\{
 g_{i}(y) 
\textstyle
+f_i \big( y,  \avg(\bs x^\star) \big)
\, | \;  y \in \mc X_i(\bs x^{\star}_{-i}) 
\right\}. \quad
\end{multline*}

\vspace*{-1.2em}
{\hfill $\square$}
\end{definition}

\medskip
We note that Nash and aggregative equilibria are strictly connected. In fact, under some mild assumptions, it can be proven that every GAE equilibrium is an $\varepsilon$-GNE equilibrium, with $\varepsilon$ vanishing as $N$ diverges \cite[\S 4]{paccagnan2018nash}. Thus, in large-scale games where the agents are unaware of the population size, e.g. \cite{deori2018price}, 
a GAE represents a good approximation of a GNE.

\subsection{Variational equilibria and pseudo-subdifferential mapping} \label{subsec:VA}

In this paper, we focus on the subclass of \textit{variational} GNE (v-GNE) that corresponds to the solution set of an appropriate generalized variational inequality, i.e., GVI$(P, \bs{\mc X})$, namely, the problem of finding $\bs x^* \in\ \bs{\mc X}$ such that
\begin{align*}
\langle \bs z^*, \bs x -\bs x^*\rangle \geq 0, \quad \forall \bs x  \in \bs {\mc X}, \, \bs z^* \in P(\bs x^*),
\end{align*}
where the mapping $P: \R^{nN} \rightrightarrows \R^{nN}$ denotes the so-called \textit{pseudo-subdifferential} (PS) of the game in \eqref{eq:Game}, defined as
\begin{align}  \label{eq:PsGr}
 P(\bs x) :=& \textstyle
\prod_{i =1}^N \partial_{x_i} \, J_i \left( x_i, \,  \bs x_{-i} \right).
\end{align}
Namely, the mapping $P$ is obtained by stacking together the subdifferentials of the agents' cost functions with respect to their local decision variables.
Given the splitting structure of the cost functions in \eqref{eq:CFi}, it follows by invoking \cite[Cor. 16.48 (iii)]{bauschke2017convex} component-wise that the PS can be written as the sum of a set-valued mapping and a single-valued one:
$$P = G + F,$$
where
\begin{align} \label{eq:G}
 G(\bs x) &:= \textstyle
  \prod_{i =1}^N \partial g_i(x_i), \\[.2em]
F(\bs x) &:= \textstyle
\col
  \left( { \big\{ \nabla_{x_i}f_i(x_i,\avg(\bs x)) \big\} }_{i=1}^N \right).
  \label{eq:F}
\end{align}
Note that, since the local decision variable $x_i$ of agent $i$ enters also in the second argument of the cost function $f_i(\cdot, \frac{1}{N} \cdot + \frac{1}{N}\sum_{j \neq i} x_j)$, with Leibniz notation, we have that
\begin{multline}
\label{eq:grad_agg}
\textstyle
\nabla_{x_i}f_i(x_i,\avg(\bs x)) \\[.2em]
= \textstyle
\left( \nabla_{x_i} f_i(x_i,z) + \frac{1}{N} \nabla_{z}f(x_i,z)
\right)
\hspace*{-0.2em}
\big|_{z=\avg(\bs x)}.
\end{multline}
In the remainder of the paper, let us refer to $F$ as \textit{pseudo-gradient} mapping (with a little abuse of terminology).

\smallskip
Under Assumption\blue{s \ref{ass:CFn}} and \ref{ass:CCF}, it follows by \cite[Prop.~12.4]{palomar2010convex} that any solution to GVI$(P,\bs{\mc X})$ is a (variational) Nash equilibrium of the game in \eqref{eq:Game}. The inverse implication is not true in general, and actually in passing from the Nash equilibrium problem to the GVI problem most solutions are lost \cite[\S~12.2.2]{palomar2010convex}; indeed, a game may have a Nash equilibrium while the corresponding GVI has no solution. 
Note that, if $J_i$ in \eqref{eq:CFi} is continuously differentiable for all $i \in \mc I$, then $P$ is a single-valued mapping and GVI$(P,\bs{\mc X})$ reduces to VI$(P,\bs{\mc X})$, which is commonly addressed in the context of game theory via projected pseudo-gradient algorithms, e.g. \cite{koshal:nedic:shanbhag:16,paccagnan2018nash,belgioioso2018projected}.

\smallskip
Next, we assume monotonicity of the PS mapping $P$, which ``is one of the weakest conditions under which global convergence can be proved" for VI-type methods  \cite[\S~5.2]{facchinei2010generalized}.

\smallskip
\begin{assumption}[Monotone and Lipschitz pseudo-gradient] \label{ass:monot}
The mapping $F$ in \eqref{eq:F} is maximally monotone and $\ell-$Lipschitz continuous over $\bs \Omega:= \prod_{i \in \mc I}  \Omega_i$, for some $\ell>0$.
{\hfill $\square$}
\end{assumption}

\smallskip
It directly follows that also the PS $P$ is maximally monotone since it is the sum of two maximally monotone operators \cite[Cor. 25.5]{bauschke2017convex}, i.e., $P=G+F$, where $G$ is maximally monotone as concatenation of maximally monotone operators \cite[Prop. 20.23]{bauschke2017convex} (i.e., the subdifferentials of the continuous and convex functions $g_i$'s \cite[Th. 20.25]{bauschke2017convex}), and $F$ is maximally monotone by Assumption \ref{ass:monot}.

The following lemma recalls some sufficient conditions for the existence and uniqueness of a variational GNE (v-GNE).
\begin{lemma}[Existence and Uniqueness of v-GNE]
Let Assumption \ref{ass:CCF} be satisfied. The following hold:
\begin{enumerate}[(i)]
\item If $\Omega_i$ is bounded, for all $i \in \mc I$, and $P$ is (strictly) monotone then there exists a (unique) solution to GVI$(P,\bs{\mc X})$.
\item If $P$ is strongly monotone then there exists a unique solution to GVI$(P,\bs{\mc X})$.
\end{enumerate}
\end{lemma}
\begin{proof}
(i) \cite[Prop.~23.36]{bauschke2017convex}; (ii) \cite[Cor.~23.37]{bauschke2017convex}.
\end{proof}

\smallskip
Hereafter, we assume that a v-GNE of the game in \eqref{eq:Game} exists.

\smallskip
\begin{assumption}[Existence of a v-GNE] \label{ass:vGNEex}
The set of solutions to GVI$(P,\bs{\mc X})$ is nonempty.
{\hfill $\square$}
\end{assumption}

\smallskip 
\begin{remark}[Approximate pseudo-gradient]
Let all cost functions $(f_i)_{i \in \mathcal{I}}$ be uniformly bounded (on their respective feasible sets) for all population sizes, $N$. As the latter grows, the second term in the right hand side of \eqref{eq:grad_agg} vanishes. In fact, if $\lim_{N\rightarrow \infty}\avg(\bs x) < \infty$, we have that
\begin{equation}
\lim_{N \rightarrow \infty} \textstyle
\nabla_{x_i}f_i(x_i,\avg(\bs x))
= \textstyle
\nabla_{x_i} f_i(x_i,z)
\big|_{z= \lim_{N\rightarrow \infty}\avg(\bs x)}.
\end{equation}
Thus, let us define an approximate version of the PG in \eqref{eq:F} for large-scale games, i.e.,
\begin{equation}
\tilde{F}(\bs x) := \textstyle
\col
  \left( {\left\{\nabla_{x_i} f_i(x_i,z)
\big|_{z=\avg(\bs x)}\right\}}_{i=1}^N \right),
  \label{eq:aF}
\end{equation}
and the correspondent approximate PS, i.e.,
\begin{equation}
\label{eq:aP}
\tilde{P} := G + \tilde{F}.
\end{equation}
As for v-GNE, one can show that any solution to GVI$(\tilde{P},\bs{\mc X})$ is a (variational) GAE (v-GAE) of the game in \eqref{eq:Game} \cite{depersis:grammatico:20}. 
{\hfill $ \square$}
\end{remark}

\subsection{Nash equilibria as zeros of a monotone operator} \label{subsec:GNEZ}
In this section, we exploit operator theory to recast the Nash equilibrium problem into a monotone inclusion, namely, the problem of finding a zero of a set-valued monotone operator.
As first step, we characterize a GNE of the game in terms of KKT conditions of the inter-dependent optimization problems in \eqref{eq:Game}. For each agent $i \in \mathcal{N}$, let us introduce the Lagrangian function $ L_i$, defined as 
\begin{equation*}
L_i(\bs x,\lambda_i) := J_i(x_i,\bs x_{-i})+ \iota_{\Omega_i}(x_i) +\lambda_i^\top (A \bs x-b),
\end{equation*}
where $\lambda_i \in \R^m_{\geq 0}$ is the Lagrangian multiplier associated with the coupling constraints.
It follows from \cite[\S 12.2.3]{palomar2010convex} that the set of strategies $\bs x^*$, where the Mangasarian--Fromovitz constraint qualification holds at $x_i^*$ for the set $\mc X_i (\bs x_{-i}^*)$, for all $i \in \mc I$, is a GNE of the game in \eqref{eq:Game} if and only if there exist some dual variables $\lambda_1^*,\ldots, \lambda_N^* \in \R^m_{\ge 0}$ such that the following coupled KKT conditions are satisfied:
\begin{equation} \label{eq:KKT}
\forall i \in \mc I:
\begin{cases}
0 \in \partial_{x_i} J_i(x_i^*,\bs x^*_{-i}) + \nc_{\Omega_i}({x}^*_i) + A_i^\top \lambda_i^*\\
0 \leq \lambda_i^* \perp -(A {\bs x}^*-b) \geq 0
\end{cases}  
\end{equation}

Similarly, we characterize a v-GNE in terms of KKT conditions by exploiting the Lagrangian duality scheme for the corresponding GVI problem, see \cite[\S 3.2]{auslender2000lagrangian}. Specifically, if $\bs {\mc X}$ satisfies the Slater's condition (Assumption 1), it follows by \cite[Th.\ 3.1]{auslender2000lagrangian} that ${\bs x}^*$ is a solution to GVI$(\bs{\mc X}, P)$ if and only if there exists a dual variable $\lambda^* \in \R^m_{\geq 0}$ such that
\begin{align} \label{eq:VI-KKT}
\begin{cases}
0 \in \partial_{x_i} J_i({x}^*_i,{\bs x}^*_{-i}) + \nc_{\Omega_i}({x}^*_i) + A_i^\top \lambda^*, 
\ \forall i \in \mc I
\\
0 \leq \lambda^* \perp -(A {\bs x}^*-b) \geq 0.
\end{cases}
\end{align} 

To cast \eqref{eq:VI-KKT} in compact form, we introduce the set-valued mapping $T: \bs{\Omega} \times \R^{m}_{\geq 0} \rightrightarrows \R^{nN} \times \R^{m}$, defined as
\begin{align} \label{eq:T}
T: 
\begin{bmatrix}
\bs x\\
\lambda
\end{bmatrix}
\mapsto
\begin{bmatrix}
\nc_{\bs \Omega}(\bs x)+P(\bs x)  + A^\top \lambda \\
\nc_{\R^{m}_{\geq 0}}(\lambda) - (A \bs x - b)
\end{bmatrix}.
\end{align}

The role of the mapping $T$ in \eqref{eq:T} is that its zeros correspond to the v-GNE of the game in \eqref{eq:Game}, or, equivalently, to the solutions to the KKT system in \eqref{eq:KKT} with equal dual variables, i.e., $\lambda_i = {\lambda}^*$ for all $i \in \mc I$, as formalized in the next statement.


\smallskip
\begin{proposition} \label{pr:UvGNE}
Let Assumptions \ref{ass:CFn}, \ref{ass:CCF} hold.
Then, the following statements are equivalent:
\begin{enumerate}[(i)]
\item $\bs x^*$ is a v-GNE of the game in \eqref{eq:Game};
\item $\exists \lambda^* \in \R^m_{\geq 0}$ such that, the pair $(x_i^*,\lambda^*)$ is a solution to the KKT in \eqref{eq:KKT}, for all $i \in \mc I$; 
\item $\bs x^*$ is a solution to GVI$(P,\bs{\mc X})$;
\item $\exists \lambda^* \in \R^m_{\geq 0}$ such that $\col(\bs x^*, \lambda^*) \in \zer (T)$.
{\hfill $\square$}
\end{enumerate}
\end{proposition}

\begin{proof}
The equivalence (i)$\Leftrightarrow$(iii) is proven in \cite[Prop.\ 12.4]{palomar2010convex}. (iii)$\Leftrightarrow$(iv) follows by \cite[Th.\ 3.1]{auslender2000lagrangian}. (iv)$\Leftrightarrow$(ii) follows by noting that \eqref{eq:KKT}, with $\lambda_1^* = \ldots \lambda_N = \lambda^*$, is equivalent to \eqref{eq:VI-KKT}, whose solutions corresponds to the zeros of $T$ \cite[\S 3.2]{auslender2000lagrangian}.
\end{proof}

\smallskip
A similar equivalence can be derived for v-GAE.

\begin{proposition} \label{pr:UvGAE}
Let Assumption\blue{s \ref{ass:CFn}}, \ref{ass:CCF} hold.
Then, the following statements are equivalent:
\begin{enumerate}[(i)]
\item $\bs x^\star$ is a v-GAE of the game in \eqref{eq:Game};
\item $\exists \lambda^\star \in \R^m_{\geq 0}$ such that, the pair $(x_i^\star,\lambda^\star)$ is a solution to the KKT in \eqref{eq:KKT} with $\partial_{x_i} J_i(x_i^\star,\bs x_{-i}^\star)$ replaced by $\partial g_i(x_i^\star) +\nabla_{x_i} f_i(x_i^\star,z)
\big|_{z=\avg(\bs x^\star)}$,
for all $i \in \mc I$; 
\item $\bs x^\star$ is a solution to GVI$(\tilde{P},\bs{\mc X})$;
\item $\exists \lambda^\star \in \R^m_{\geq 0}$ such that $\col(\bs x^\star, \lambda^\star) \in \zer (\tilde T)$, where $\tilde T$ is analogous to $T$ in \eqref{eq:T} with $P$ replaced by its approximation $\tilde P$ in \eqref{eq:aP}.
{\hfill $\square$}
\end{enumerate}
\end{proposition}
\begin{proof}
The proof is similar to that of Proposition \ref{pr:UvGNE}.
\end{proof}


\section{Generalized Nash equilibrium seeking: Operator-theoretic characterization} \label{sec:CREA}

\subsection{Zero finding methods for GNE seeking}
In Section \ref{subsec:GNEZ}, we show that the original GNE seeking problem corresponds to the following generalized equation:
\begin{equation} \label{eq:Minc}
\text{find } \bs \omega^*:=\col(\bs x^*, \lambda^*) \in \zer(T).
\end{equation}
Next, we show that the mapping $T$ can be written as the sum of two operators, i.e., $T = T_1 + T_2$, where
\begin{align} \label{eq:T1}
T_1: 
\bs \omega &\mapsto
\col(F(\bs x),b);
\\
\label{eq:T2}
T_2: 
\bs \omega  &\mapsto 
\big(\nc_{\bs \Omega}(\bs x) + G(\bs x)\big) \times
\nc_{\R^m_{\geq 0}}(\lambda) 
 +S \bs \omega
\end{align}
and $S$ is a skew symmetric matrix, i.e., $S^\top = -S$,  defined as
\begin{align} \label{eq:SandPhi}
S:=
\begin{bmatrix}
0 & A^\top \\
-A & 0
\end{bmatrix}.
\end{align}
The formulation $T = T_1 + T_2$ is called \textit{splitting} of $T$, and we exploit it in different ways later on. We show next that the mappings $T_1$ and $T_2$ are both maximally monotone, which paves the way for operator splitting algorithms \cite[\S~26]{bauschke2017convex}.

\smallskip
\begin{lemma} \label{lem:U-Bmon}
Let Assumption\blue{s \ref{ass:CFn},} \ref{ass:CCF}, \ref{ass:monot} hold.
The mappings $T_1$ in \eqref{eq:T1}, $T_2$ in \eqref{eq:T2}  and $T$ in \eqref{eq:T} are maximally monotone.
{\hfill $\square$}
\end{lemma}
\begin{proof}
$T_1$ is maximally monotone since $F$ is such by Assumption \ref{ass:monot}, $b$ is a constant, thus maximally monotone, and the concatenation of maximally monotone operator remains maximally monotone \cite[Prop.~20.23]{bauschke2017convex}. 
The first term of $T_2$, i.e., $(\nc_{\bs \Omega} + G\big) \times
\nc_{\R^m_{\geq 0}} $, is maximally monotone, since normal cones of closed convex sets are maximally monotone and the concatenation preserves maximality \cite[Prop.~20.23]{bauschke2017convex}; the second term, i.e., $S$, is linear and skew symmetric, i.e., $S^\top = - S$, thus maximally monotone \cite[Ex.~20.35]{bauschke2017convex}. Then, the sum of the previous terms, namely, $T_2$, is maximally monotone by \cite[Cor.~25.5]{bauschke2017convex}, since $\dom S = \R^{nN+m}$. Equivalently, the maximal monotonicity of $T = T_1 + T_2$ follows from \cite[Cor.~25.5]{bauschke2017convex}, since $\dom T_1 = \R^{nN+m}$.
\end{proof}

\medskip
In the remainder of this section, we characterize the main features and limitations of some existing semi-decentralized algorithms for aggregative games with coupling constraints from a general operator-theoretic perspective.  

\smallskip

\begin{remark}[Generalized aggregative equilibrium seeking]
In light of Proposition \ref{pr:UvGAE}, the same operator-theoretic approach can be exploited to recast the GAE seeking problem as a monotone inclusion problem. It follows that all the GNE seeking algorithms introduced next can be adopted for seeking a GAE. Specifically, for gradient-based algorithms, it is sufficient to replace $\nabla_{x_i}f_i(x_i,\avg(\bs x))$ in \eqref{eq:F} with its approximate version, i.e., $\textstyle
\nabla_{x_i} f_i(x_i,z)
\big|_{z=\avg(\bs x)}$. 
{\hfill $\square$}
\end{remark}

%
%
%
%
%
%


\subsection{Preconditioned forward-backward algorithm}
\label{subsec:pFB}
The main idea of the preconditioned forward-backward algorithm (pFB, Algorithm 1) is that the zeros of the mapping $T$ in \eqref{eq:T} correspond to the fixed points of a certain operator which depends on the chosen splitting \eqref{eq:T1}$-$\eqref{eq:T2} \cite[\S 26.5]{bauschke2017convex} and on an arbitrary symmetric, positive definite matrix $\Phi$, known as preconditioning matrix \cite{yi2017distributed}. The pFB method, proposed in \cite{yi2017distributed} for strongly monotone games, is applicable to aggregative games with semi-decentralized algorithmic structure \cite{belgioioso2018projected}, in which case it reduces to the APA \cite[\S 12.5.1]{facchinei:pang}, also proposed in \cite{paccagnan2018nash}. A critical assumption for the convergence of this method is the cocoercivity of the pseudo-gradient mapping $F$ in \eqref{eq:F}, as postulated next.

\smallskip
\begin{assumption}[Cocoercive pseudo-gradient] \label{ass:coco}
The mapping $F$ in \eqref{eq:F} is $\gamma-$cocoercive on $\bs \Omega$, for some $\gamma>0$.
{\hfill $\square$}
\end{assumption}

\smallskip
\begin{remark}[\blue{Sufficient conditions for cocoercivity of $F$}] \label{rem:SMONandLIP_coco}
If $F$ is $\mu-$strongly monotone and $\ell-$Lipschitz, $\ell\geq\mu > 0$, then $F$ is $(\mu/\ell^2)-$cocoercive. On the contrary, cocoercive mappings are not necessarily strongly monotone, e.g. the gradient of a non-strictly convex smooth function. \blue{Some sufficient conditions for cocoercivity of $F$ based on the local cost functions $(f_i)_{i \in \mc I}$ are provided in Appendix \ref{cond:suffCOCO}.}
{\hfill $\square$}
\end{remark}

\begin{figure}[b]
\hrule
\smallskip
\textsc{Algorithm $1$}: Preconditioned forward-backward (pFB)
\smallskip
\hrule
\medskip

\textbf{Initialization}: $\delta > \frac{1}{2\gamma}$;
$\forall i \in \mc I$, $x_i^0 \in \R^{n}$, $0<\alpha_i \leq \textstyle (\norm{A_i} + \delta)^{-1}$; $\lambda^0 \in \R^m_{\geq 0}$, $0<\beta \leq ( \frac{1}{N} \sum_{i=1}^N \norm{A_i} +\frac{1}{N} \delta )^{-1}$.

\medskip
\noindent
\textbf{Iterate until convergence}:
\begin{enumerate}[1.]
\item Local: Strategy update, for all $i \in \mc I$:\\[.5em]
\hspace*{1em} $
\begin{array}{l}
y_i^{k} = x_i^k -  \alpha_i ( \nabla_{x_i}f_i(x_i^{k},\avg(\bs x^{k})) + A_i^\top \lambda^k)\\[.4em]
x_i^{k+1} = \prox_{\alpha_i g_i + \iota_{\Omega_i}}\,(y_i^{k})\\[.4em]
d_i^{k+1} = 2A_i x_i^{k+1} - A_i x_i^k - b_i
\end{array}
$\\[.5em]
%

\item Central coordinator: dual variable update\\[.5em]
\hspace*{1em} $
\begin{array}{l}
\lambda^{k+1} = \proj_{\R^m_{\geq 0}} \big(\lambda^k + \beta \, \avg(\bs d^{k+1}) \big)
\end{array}
$\\
\end{enumerate}
\hrule
\end{figure}

\begin{remark}(i) The local auxiliary variables $y_i$'s and $d_i$'s are introduced to cast Algorithm 1 in a more compact form. The quantity $\avg( \bs d^{k+1}) := \frac{1}{N} \sum_{i=1}^N (2 A_i x_i^{k+1} - A_i x_i^k - b_i) $ measures the violation of the coupling constraints, technically, it is the ``reflected violation" of the constraints at iteration $k$.\\
(ii) The proximal operator in Algorithm $1$ reads as
\begin{equation*} \textstyle
\prox_{\alpha_i g_i + \iota_{\Omega_i}}(y) =
\left\{
\begin{array}{r l}
\underset{z \in \R^n}{\argmin} \,&  g_i(z)+\frac{1}{2\alpha_i}\|z-y\|^2\\
\mathrm{s.t.} & z \in  \Omega_i
\end{array}
\right.
\end{equation*}
If $g_i = 0$, then the primal update in Algorithm $1$ becomes a projection, i.e., $\prox_{\alpha_i g_i + \iota_{\Omega_i}} = \proj_{\Omega_i}$.
{\hfill $\square$}
\end{remark}

\smallskip
If Assumption \ref{ass:coco} holds and the step sizes $\{\alpha_i \}_{i \in \mc I}$ and $\beta$ are small enough, then the sequence $( \col(\bs x^k, \lambda^k ))_{k \in \bb N}$ generated by Alg. $1$ converges to some $\col(\bs x^*, \lambda^*) \in \zer(T)$, where $\bs x^*$ is a v-GNE, see \cite[Th.\ 1]{belgioioso2018projected} for a formal proof of convergence.

\smallskip
Algorithm $1$ is semi-decentralized. In fact, at each iteration $k$, a central coordinator is needed to:
\begin{enumerate}[(i)]
\item gather and broadcast the average strategy $\avg(\bs x^k)$;
\item gather the reflected violation of the constraints $\avg(\bs d^k)$;
\item update and broadcast the dual variable $ \lambda^k$.
\end{enumerate}
Specifically, after each central and local update in Algorithm $1$, a communication stage follows. The central coordinator broadcasts to all the agents the current values of the aggregate function $\avg(\bs x^k)$ and the multiplier vector $\lambda^k$. In return, each agent $i \in \mc I$ updates its own strategy $x_i$, based on the received
signals, and forwards it to the central coordinator. Moreover, at each iteration only two vectors, in $\R^n$ and $\R^m$ respectively, are broadcast, independently on the population size $N$. Each decentralized computation consists of solving a finite-dimensional convex optimization problem, for which efficient algorithms are available.

\smallskip
\begin{remark}
The primal-dual iterations of Algorithm $1$ are sequential, namely, while the local primal updates $x_i^{k+1}$ can be performed in parallel, the dual update, $\lambda^{k+1}$, exploits
the most recent value of the agents' strategies, $d_i^{k+1}$. This feature is convenient since it follows the natural information flow in the considered semi-decentralized communication structure.
{\hfill $\square$}
\end{remark}

\smallskip
\subsubsection*{Algorithm $1$ as a fixed-point iteration}
The dynamics generated by Algorithm $1$ can be cast in a compact form as the fixed-point iteration
\begin{align} \label{eq:R}
\bs \omega^{k+1}=
R_{\text{FB}}( \bs \omega^k),
\end{align}
where $\bs \omega^k = \col(\bs x^k,  \lambda^k)$ is the vector of the primal-dual variables and $R_{\text{FB}}$ is the so-called FB operator \cite[Eq. (26.7)]{bauschke2017convex}:
\begin{align} \label{eq:Rmap}
R_{\text{FB}}:= (\Id+\Phi^{-1} T_2)^{-1}\circ(\Id-\Phi^{-1} T_1),
\end{align}
where $T_1$ and $T_2$ as in \eqref{eq:T1}$-$\eqref{eq:T2} and $\Phi$ is a preconditining matrix, here defined as
\begin{align} \label{eq:Phi}
\Phi := \begin{bmatrix}
\bar{\alpha}^{-1} \otimes I_n &  - A^\top \\
 -A &  N\beta^{-1} I_m
\end{bmatrix},
\end{align}
with $\bar{\alpha}=\diag(\alpha_1,\ldots,\alpha_N)$. When the mapping $T_1$ is cocoercive (Assumption \ref{ass:coco}), $T_2$ is maximally monotone (Lemma \ref{lem:U-Bmon}) and the step sizes in the main diagonal of $\Phi$ are set as in Algorithm $1$, then the preconditioned mappings $\Phi^{-1}T_1$ and $\Phi^{-1}T_2$ satisfy the following properties with respect to the $\Phi-$induced norm (\cite[Lemma 7]{yi2019operator}):
\begin{enumerate}[(i)]
\item $\Phi^{-1}T_1$ is $\gamma \delta-$cocoercive w.r.t. $\|\cdot \|_{\Phi}$,

\item $\Phi^{-1}T_2$ is maximally monotone w.r.t. $\|\cdot \|_{\Phi}$.
\end{enumerate}
It follows from \cite[Prop. 26.1(iv)-(d)]{bauschke2017convex} that the FB operator $R_{\text{FB}}$ in \eqref{eq:Rmap} is averaged with respect to the same norm, i.e., 
\begin{enumerate}
\item[(iii)]  $R_{\text{FB}}$ is $\left( \frac{2\delta \gamma}{4\delta \gamma-1}\right)-$averaged w.r.t. $\|\cdot \|_{\Phi}$.
\end{enumerate}
Hence, the Banach--Picard fixed-point iteration in \eqref{eq:R} converges to some $\bs \omega^*:= \col(\bs x^*,  \lambda^*) \in \fix(R_{\text{FB}})$ \cite[Prop. 5.16]{bauschke2017convex}, where $\fix(R_{\text{FB}})= \zer(T)$ \cite[Prop. 26.1(iv)-(a)]{bauschke2017convex}, $\zer(T) \neq \varnothing$ (Assumption \ref{ass:vGNEex}) and, therefore, $\bs x^*$ is a v-GNE by Prop. \ref{pr:UvGNE}. We refer to \cite{yi2019operator,belgioioso2018projected} for a complete convergence analysis.

\smallskip
\subsubsection*{Inertial pFB algorithm}
To conclude this section, we recall the inertial version of the pFB (Algorithm $1$), originally proposed for the more general context of generalized network games in \cite[Alg.~2]{yi2019operator} and summarized here in Algorithm 1B.

\begin{figure}[b]
\hrule
\smallskip
\textsc{Algorithm 1B}: Inertial pFB (I-pFB)
\smallskip
\hrule
\medskip

\textbf{Initialization}: $\theta \in [0,1/3)$ and $\delta > \frac{(1-\theta)^2}{2\gamma(1-3\theta)}$, with $\gamma$ as in Assumption \ref{ass:coco}; for all $ i \in \mc I$, $x_i^0 = \tilde x_i^0 \in \R^{n}$, $0 < \alpha_i \leq \textstyle (\norm{A_i} + \delta)^{-1}$; $\lambda^0 = \tilde \lambda^0 \in \R^m_{\geq 0}$, $0<\beta \leq ( \frac{1}{N} \sum_{i=1}^N \norm{A_i} + \frac{1}{N} \delta )^{-1}$.

\medskip
\noindent
\textbf{Iterate until convergence}:
\begin{enumerate}[1.]
\item Local: Strategy update, for all $i \in \mc I$:\\[.5em]
\hspace*{1em} $
\begin{array}{l}
 y_i^{k} = \tilde x_i^k -  \alpha_i ( \nabla_{x_i}f_i(\tilde x_i^{k},\avg(\bs{\tilde x}^{k})) + A_i^\top \tilde \lambda^k)\\[.4em]
 {x}_i^{k+1} = \prox_{\alpha_i g_i + \iota_{\Omega_i}}\,(y_i^{k})\\[.4em]
 \tilde{x}_i^{k+1} = x^{k+1}_i + \theta( x^{k+1}_i- x^{k}_i)\\[.4em]
d_i^{k+1} = 2A_i x_i^{k+1} - A_i \tilde x_i^k - b_i
\end{array}
$\\[.5em]

\item Central coordinator: Dual variable update:\\[.5em]
\hspace*{1em} $
\begin{array}{l}
\lambda^{k+1} = \proj_{\R^m_{\geq 0}} \big(\tilde \lambda^k + \beta \, \avg(\bs d^{k+1}) \big)\\[.4em]
\tilde{\lambda}^{k+1} = {\lambda}^{k+1} + \theta( {\lambda}^{k+1}- {\lambda}^{k})
\end{array}
$\\[.4em]

\end{enumerate}
\hrule
\end{figure}

\smallskip
We note that the inertial extrapolation phase, at the end of the local and central updates, improves the converge properties of the pFB algorithm. The convergence of Algorithm 1B can be studied via fixed-point theory \cite{mainge2008convergence}, or by relying on the inertial version of the FB splitting method \cite{lorenz2015inertial}. We refer to \cite[Th.~2]{yi2019operator} for a complete convergence proof of this algorithm.

\smallskip
\subsubsection*{Algorithm 1B as a fixed-point iteration}
The dynamics generated by Algorithm 1B can be cast in a compact form as the following inertial fixed-point iteration:
\begin{subequations} \label{eq:Inertial-R}
\begin{align} 
\tilde{\bs \omega}^k &= \bs \omega^k + \theta(\bs \omega^k- \bs \omega^{k-1}),\\
\bs \omega^{k+1}&=
R_{\text{FB}}( \tilde{\bs \omega}^k),
\end{align}
\end{subequations}
where $\bs \omega^k = \col(\bs x^k,  \lambda^k)$ and $\tilde{\bs \omega}^k = \col(\tilde{\bs x}^k,  \tilde \lambda^k)$ are the stacked vectors of the iterates and $R_{\text{FB}}$ is the FB operator defined in \eqref{eq:Rmap}. The convergence analysis of inertial schemes as in \eqref{eq:Inertial-R} are studied in \cite{mainge2008convergence}; while more precise conditions for the convergence of \eqref{eq:Inertial-R} are derived in \cite[Th.~1]{lorenz2015inertial}.

\subsection{Algorithms for (non-strictly) monotone aggregative games} \label{sec:AMAG}

When the pseudo-gradient mapping $F$ is non-cocoercive, non-strictly monotone, then Algorithm $1$ may fail to converge, see \cite{grammatico:18} for an example of non-convergence.
Few algorithms are available in the literature for solving merely monotone (aggregative) games with coupling constraints, each with important technical or computational limitations. 

\smallskip
\subsubsection*{Iterative Tikhonov regularization (\textsc{Algorithm $2$})}
To be applicable to aggregative games with (non-cocoercive, non-strictly) monotone pseudo-gradient mapping, the forward-backward algorithm should be augmented with a vanishing regularization. This approach is known as iterative Tikhonov regularization (ITR) and generates a forward-backward algorithm with double-layer vanishing step sizes \cite[\S 1.3(a)]{kannan2012distributed}:
\begin{align*}
&  \forall i:
\left\{
\begin{array}{l}
y_i^{k} = x_i^k -  \gamma^k ( \nabla_{x_i}f_i(x_i^{k},\avg(\bs x^{k})) + A_i^\top \lambda^k + \epsilon^k x_i^k )\\[.2em]
 x_i^{k+1} = \proj_{{\Omega_i}}\,(y_i^{k}), \quad d_i^{k+1} = A_i x_i^{k+1} - b_i
\end{array}
\right.
\\
&\lambda^{k+1} = \proj_{\R^m_{\geq 0}} \big(\lambda^k + \gamma^k (N \avg(\bs d^{k}) - \epsilon^k \lambda^k) \big).
\end{align*}

The convergence proof is based on the fact that the actual step size $\gamma^k$ must vanish faster than the vanishing regularization parameter $\epsilon^k$, \cite[(A2.2), \S 2.1]{kannan2012distributed}. 
The extension of ITR schemes to non-smooth games can be possibly achieved by discretizing the algorithm proposed in \cite[\S~4]{boct2020inducing}.

\smallskip
\subsubsection*{Inexact preconditioned proximal-point (\textsc{Algorithm $3$})}
Recently, the inexact preconditioned proximal-point (PPP) method \cite{yi2018distributed} was proposed to solve monotone (aggregative) games, virtually with no additional technical assumption other than monotonicity of the PS mapping $P$. 
When applied to the game in \eqref{eq:Game}, the PPP \cite[Alg. 2]{yi2018distributed} generates a double-layer algorithm, in which at each (outer) iteration $k$, the inner loop consists of solving (inexactly) an aggregative game without coupling constraints and with cost functions $\bar J_i$'s defined as 
\begin{align*}
\bar J^k_i\left( x_{i}, { \bs{x}}_{-i} \right) = J_i\left( x_{i}, \bs{x}_{-i} \right) + {(A_i^\top \lambda^k)}^\top x_i + \alpha_i {\| x_i - x_i^k\|}^2,
\end{align*}
where $\lambda^k$, i.e., the dual variable, and $x_i^k$, i.e., the so-called centroid, stay fixed during the inner iterations. When the subgame is solved with the desired precision $\varepsilon^k$, namely, an $\varepsilon^k-$NE profile $\bar{ \bs x}^k$ is reached, the agents update their centroids
\begin{align*}
(\forall i \in \mc I): \quad x_i^{k+1} = \bar x_i^k, \quad d_i^{k+1} = 2 A_i x_i^{k+1} - A_i x_i^k - b_i.
\end{align*}
Finally, the central coordinator updates the dual variable as
\begin{align*}
\lambda^{k+1} = \textstyle
\proj_{\R^m_{\geq 0}}(
\lambda^k + \beta \, \avg(\bs d^{k+1})).
\end{align*}
The primal-dual dynamics generated by the PPP can be cast in compact form as the fixed-point iteration
\begin{align} \label{eq:J}
\bs \omega^{k+1}= \text J_{\Phi^{-1} T}
( \bs \omega^k) + e^k,
\end{align}
where $\bs \omega^k = \col(\bs x^k,  \lambda^k)$ is the vector of primal-dual iterates, $J_{\Phi^{-1} T}$ is the so-called resolvent operator of the mapping $\Phi^{-1} T$, defined as
$
J_{\Phi^{-1} T} := (\Id+\Phi^{-1} T)^{-1},
$
and $e^k$ is an error term that accounts for the inexact computations of $\text J_{\Phi^{-1} T}
( \bs \omega^k)$.

When the mapping $T$ is maximally monotone (Lemma \ref{lem:U-Bmon}) and the step sizes in the main diagonal of $\Phi$ are set such that $\Phi \succ 0$, then the resolvent $J_{\Phi^{-1} T}$ is firmly nonexpansive \cite[Prop. 23.8]{bauschke2017convex} ($1/2-$averaged) w.r.t. the $\Phi-$induced norm, i.e., $\|\cdot \|_{\Phi}$. Moreover, if the error sequence $(e^k)_{k \in \bb N}$ is summable (which is guaranteed by solving the regularized sub-games with increasing precision, namely $\sum_{k=0}^\infty \varepsilon_k < \infty$), then the inexact fixed-point iteration \eqref{eq:J} converges to some $\bs \omega^*:= \col(\bs x^*, \lambda^*) \in \fix(J_{\Phi^{-1} T}) = \zer(T) \neq \varnothing$ \cite[Prop. 5.34]{bauschke2017convex}, where $\bs x^*$ is a v-GNE. We refer to \cite{yi2018distributed} for a complete convergence analysis of Algorithm $3$.

\smallskip
\begin{remark}[Computational limitations of ITR and PPP]
The solution of each sub-game of the PPP (Algorithm 3), requires nested (inner) iterations, and, therefore, multiple communication stages between the agents and the central coordinator. Similarly to the (doubly) vanishing step sizes of the ITR schemes (Algorithm 2), that lead to slow speed of convergence in practice, we can regard double-layer or nested iterations as an important computational limitation.
{\hfill $\square$}
\end{remark}

\smallskip
\subsubsection*{Tseng's forward-backward-forward splitting}
To solve non-cocoercive, non-strictly monotone aggregative games via non-vanishing iterative steps or nested iterations, the forward-backward-forward (FBF) method \cite[\S 26.6]{bauschke2017convex} adds an additional forward step to the FB algorithm.
In Algorithm 4, we introduce a modified version of the FBF algorithm for aggregative games, originally proposed in \cite[Alg. 1]{belgioioso:grammatico:17cdc}.

\smallskip
Algorithm $4$ improves \cite[Alg. 1]{belgioioso:grammatico:17cdc} on two main aspects:
\begin{enumerate}[(i)]
\item (Partially uncoordinated step sizes) each agent $i \in \mc I$ and the central coordinator have decision authority on their own local step sizes;

\item (Additional projection) The local updates in step 3 and  the central update in step 4 are projected onto the local feasible sets, $ \Omega_i$'s and $\R^m_{\geq 0}$, respectively. These additional projections make sure that the iterates $x_i^k$'s live in the domain of correspondent functions $f_i$'s and in fact can improve the convergence speed of the algorithm.
\end{enumerate}

\begin{figure}[htbp]
\hrule
\smallskip
\textsc{Algorithm $4$}: Tseng's forward-backward-forward (FBF)
\smallskip
\hrule
\medskip

\textbf{Initialization}: For all $ i \in \mc I$, $x_i^0 \in \R^{n}$ and $0 <\alpha_i < (\ell+ \|A\|)^{-1}$; $\lambda^0 \in \R^m_{\geq 0}$ and $0< \beta < (\ell+ \|A\|)^{-1}$.

\bigskip
\noindent
\textbf{Iterate until convergence}:
\begin{enumerate}[1.]
\item Local: Strategy update, for all $i \in \mc I$:\\[.5em]
\hspace*{1em}
$
\begin{array}{l}
y_i^{k} = x_i^k -  \alpha_i ( \nabla_{x_i}f_i(x_i^{k},\avg(\bs x^{k})) + A_i^\top \lambda^k)\\[.4em]
 \tilde x_i^{k} = \prox_{\, \alpha_i g_i + \iota_{\Omega_i}}\,(y_i^{k})\\[.4em]
\tilde d_i^k = A_i x_i^k -  b_i
\end{array}
$\\

\item Central coordinator: dual variable update\\[.5em]
\hspace*{1em}
$
\begin{array}{l}
\tilde \lambda^{k} = \proj_{\R^m_{\geq 0}} \big(\lambda^k + \beta \, \avg(\tilde{\bs d}^{\,k})\big)
\end{array}
$\\

\item Local: Strategy update, for all $i \in \mc I$:\\[.5em]
\hspace*{1em}
$
\begin{array}{l}
r_i^{k+1} =  \tilde{x}_i^k - \alpha_i (\nabla_{x_i}f_i(\tilde x_i^{k},\avg( \tilde{\bs x}^{k})) + A_i^\top \tilde \lambda^k)\\[.4em]
 x_i^{k+1} = \proj_{\Omega_i }(x_i^k - y_i^k + r_i^{k+1})\\[.4em]
 d_i^{k+1} = A_i \tilde{x}_i^k-b_i
\end{array}
$\\

\item Central coordinator: dual variable update\\[.5em]
\hspace*{1em}
$
\begin{array}{l}
\lambda^{k+1} = \proj_{\R^m_{\geq 0}} 
\big(
\tilde{\lambda}^k + \beta (\avg(\bs d^{k+1}) - \avg(\tilde{\bs d}^{k})) 
\big)
\end{array}
$

\end{enumerate}
\hrule
\end{figure}

\renewcommand{\arraystretch}{1.5}
\begin{table*}[htbp]
\caption{Comparison among v-GNE algorithms. Legend: C stands for coordinated step sizes, P-UC for partially uncoodinated, F-UC for fully uncoordinated; MON for monotone, SMON for strongly monotone, COCO stands for (monotone and) cocoercive. \label{tab:1}}
\centering
\begin{tabular}{p{3.2cm} p{1.4cm} p{1.4cm} p{1.4cm} p{1.4cm} p{1.4cm} p{2.1cm} p{2cm}}
\hline
& APA \newline \cite[Alg.\ 2]{paccagnan2018nash}  & pFB \cite{belgioioso2018projected} \newline (Alg.\ 1)  & ITR \cite{kannan2012distributed} \newline (Alg. 2) & PPP \cite{yi2018distributed} \newline (Alg. 3) & FBF \cite{belgioioso:grammatico:17cdc} \newline (Alg. 4) & FoRB\newline (Alg. 5, $\theta=0$) & cPPP \newline (Alg. 6, $\theta^k =0$)\\ 
\hline
Communications / iteration & $1$ & $1$ & $1$ &  $\infty$ & $2$ &  $1$ &  $1$ \\
\hline
Local step sizes & fixed, \newline C & fixed, \newline P-UC & vanishing, \newline P-UC & fixed,\newline P-UC & fixed, \newline P-UC & fixed, \newline P-UC & fixed, \newline F-UC\\
\hline
Pseudo-subdifferential & SMON & COCO  & MON  & MON  & MON & MON & MON\\
\hline
Local cost functions &  diff. & non-diff. \newline as in \eqref{eq:CFi}  & diff.  &  non-diff. \newline as in \eqref{eq:CFi}  &  non-diff. \newline as in  \eqref{eq:CFi}  &  non-diff. \newline as in \eqref{eq:CFi} &  linear coupling \newline as in \eqref{eq:CF-SS}  \\
\hline
Inertia  & & Alg. 1B &    & \checkmark & \cite[Th. 4]{boct2016inertial} & \checkmark & \checkmark \\
\hline
Alternating inertia & & Cor. 1   &  & \checkmark &  &  & Cor. 2\\
\hline
Over-relaxation & & \checkmark  &  & \checkmark &  &  & Alg. 6B\\
\hline
Acceleration parameters & - & C & - & F-UC & C & P-UC& F-UC\\
\hline
\end{tabular}
\end{table*}

The convergence analysis of Algorithm $4$ is (almost) identical to that of \cite[Alg. 1]{belgioioso:grammatico:17cdc}, thus we discuss it briefly next.

\smallskip
\textit{Algorithm $4$ as a fixed-point iteration:}
In compact form, the dynamics generated by Algorithm $4$ read as 
\begin{align} \label{eq:FBF-FPI}
\bs \omega^{k+1}=
R_{\text{FBF}}( \bs \omega^k),
\end{align}
where $\bs \omega^k = \col(\bs x^k,  \lambda^k)$ is the stacked vector of the primal-dual variables and $R_{\text{FBF}}$ is the so-called FBF operator, i.e.,
\begin{multline*}
R_\text{FBF}:= \proj_{\bs \Omega \times \R^m_{\geq 0}} \circ \left( (\Id-\Psi^{-1} U_1) \right. \\
\left.
\circ \, \textrm{J}_{\Psi^{-1} U_2} \circ(\Id-\Psi^{-1} U_1) + \Psi^{-1} U_1 \right),
\end{multline*}
where $U_1$ and $U_2$ characterize an alternative splitting of the mapping $T$ in \eqref{eq:T}, i.e.,  $T=U_1+U_2$, where
\begin{align} \label{eq:U1}
U_1: 
\bs \omega &\mapsto
\col(F(\bs x),b) 
 +S \bs \omega,
\\
\label{eq:U2}
U_2: 
\bs \omega  &\mapsto 
\big(\nc_{\bs \Omega}(\bs x) + G(\bs x)\big) \times
\nc_{\R^m_{\geq 0}}(\lambda) ,
\end{align}
and $\Psi$ is the preconditining matrix, here defined as
\begin{align} \label{eq:PhiFBF}
\Psi := \begin{bmatrix}
\bar{\alpha}^{-1} \otimes I_n &  \0 \\
 \0 &  N \beta^{-1}I_m
\end{bmatrix}.
\end{align}
When the mappings $U_1$ and $U_2$ are maximally monotone (which can be proven when Assumption \ref{ass:monot} holds true by following a similar technical reasoning of that in Lemma \ref{lem:U-Bmon}), $U_1$ is Lipschitz continuous (Assumption \ref{ass:monot}) and the step sizes in the main diagonal of $\Psi$ are set small enough, then
the fixed-point iteration \eqref{eq:FBF-FPI} converges to some $\bs \omega^*:= \col(\bs x^*, \bs \lambda^*) \in \fix(R_{\text{FBF}}) = \zer(T) \cap (\bs \Omega \times \R^m_{\geq 0}) = \zer(T)$ \cite[Th.~26.17]{bauschke2017convex}, where $\bs x^*$ is a v-GNE. We refer to \cite[Th. 2]{belgioioso:grammatico:17cdc}, for a complete convergence analysis which is applicable to Algorithm $4$.

\smallskip
\begin{remark}[Double communication round]
\label{rem:2CR}
At each central and local update of Algorithm $4$ a communication takes place. Hence, each iteration of Algorithm $4$ requires two communication rounds between the agents and central operator.
{\hfill $\square$}
\end{remark}

Finally, we note that an inertial version of Algorithm 4, without the extra projections (in steps 3, 4) and with fully coordinated step sizes that match across the agents and the central coordinator can be derived based on \cite[Th. 4]{boct2016inertial}.

\section{Generalized Nash equilibrium seeking: Advanced algorithms} \label{sec:A-GNE_alg}

In this section, we design two novel semi-decentralized GNE seeking algorithms obtained by solving the monotone inclusion in \eqref{eq:Minc} with different zero-finding methods: the forward-reflected-backward splitting \cite{malitsky2020forward} and, for a particular subclass of aggregative games with linear-coupling functions, the proximal-point method with (alternated) inertia. The main features of the proposed algorithms, e.g. convergence guarantees and communication requirements, are summarized and compared with those of the existing methods in Table \ref{tab:1}.


\subsection{(Inertial) Forward-reflected-backward algorithm \label{sec:FoRB}}
In this section, we present a single-layer, single communication round algorithm for monotone generalized aggregative games that overcomes the technical and computational limitations of all the algorithms in Section \ref{sec:AMAG}. The design of the proposed method (i.e., Algorithm 5) is based on the \textit{forward-reflected-backward splitting} (FoRB) recently proposed in \cite{malitsky2020forward} to find a zero of the sum of two maximally monotone operators, one of which is single-valued and Lipschitz continuous.

\begin{figure}[h]
\hrule
\smallskip
\textsc{Algorithm $5$}: Inertial FoRB (I-FoRB)
\smallskip
\hrule
\medskip

\textbf{Initialization}: $\theta \in [0, 1/3)$ and $\delta > 2 \ell /(1-3\theta)$, with $\ell$ as in Assumption \ref{ass:monot}; $\forall i \in \mc I$, $x_i^0, x_i^{-1} \in \R^n$ and $0<\alpha_i \leq \textstyle (\norm{A_i} + \delta)^{-1}$; $\lambda^0, \lambda^{-1} \in \R^m_{\geq 0}$, $0<\beta\leq ( \frac{1}{N} \sum_{i=1}^N \norm{A_i} +\frac{1}{N}\delta )^{-1}$.

\medskip
\noindent
\textbf{Iterate until convergence}: 
\begin{enumerate}[1.]
\item Local: Strategy update, for all $i \in \mc I$:\\[.5em]
\hspace*{-1.2em}
$
\begin{array}{l}
r_i^{k } = 2\nabla_{x_i}f_i(x_i^{k},\avg(\bs x^{k})) - \nabla_{x_i}f_i(x_i^{k-1},\avg(\bs x^{k-1}))\\[.4em]
x_i^{k+1} \!= \prox_{\alpha_i g_i + \iota_{\Omega_i}} \! \big(  x_i^k -  \alpha_i ( r_i^k + A_i^\top \lambda^k) + \theta(x_i^{k} -  x_i^{k-1}) \big) \\[.4em]
%
%
d_i^{k+1} \!= 2A_i x_i^{k+1} - A_i x_i^k - b_i
\end{array}
$\\

\item Central coordinator: dual variable update\\[.5em]
\hspace*{-1.2em}
$ \begin{array}{l}
\lambda^{k+1} = \proj_{\R^m_{\geq 0}} \big(\lambda^k + \beta \, \avg(\bs d^{k+1}) + \theta (\lambda^{k} - \lambda^{k-1}) \big)
\end{array}
%
$\\
\end{enumerate}
\hrule
\end{figure}

Also for Algorithm $5$, we describe the generated dynamics as a compact inertial iteration, step
\begin{multline}  \label{eq:FP-FoRB}
\bs \omega ^{k+1} = (\Id + \Phi^{-1} T_2)^{-1}
(\bs \omega^k 
- 2 \Phi^{-1}T_1(\bs \omega ^{k})\\
+ \Phi^{-1} T_1(\bs \omega ^{k-1})
+ \theta(\bs \omega^{k} - \bs \omega^{k-1})
),
\end{multline}
where $\bs \omega^k = \col(\bs x^k, \lambda^k)$ is the stacked vector of primal-dual variables, the components mappings $T_1$, $T_2$  as in \eqref{eq:T1}$-$\eqref{eq:T2} and the preconditioning $\Phi$ as in \eqref{eq:Phi}.
If the step sizes in the main diagonal of $\Phi$ are chosen small enough, then the iteration \eqref{eq:FP-FoRB}, namely, the inertial FoRB splitting \cite [Corollary~4.4]{malitsky2020forward} on the operators $\Phi^{-1} T_1$ and $\Phi^{-1} T_2$, converges to some $\bs \omega^*:= \col(\bs x^*, \bs \lambda^*) \in \zer(T_1+T_2)=\zer(T)$, where $\bs x^*$ is a v-GNE. 

\smallskip
Our first main result is to establish global convergence of Algorithm $5$ to a v-GNE when the mapping $F$ is maximally monotone and Lipschitz continuous (Assumption \ref{ass:monot}) and the step sizes are chosen small enough.

\smallskip
\begin{theorem} \label{th:FoRB}
(\textit{Convergence of FoRB (Algorithm $5$)})
Let Assumptions \blue{\ref{ass:CFn},} \ref{ass:CCF}, \ref{ass:monot}, \ref{ass:vGNEex} hold true. The sequence $( \col(\bs x^k, \lambda^k))_{k \in \bb N}$ generated by Algorithm $5$, globally converges to some $\col(\bs x^*,\lambda^*) \in \zer(T)\neq \varnothing$, where $\bs x^*$ is a v-GNE.
{ \hfill $\square$ }
\end{theorem}
\begin{proof}
See Appendix \ref{proof:th:FoRB}.
\end{proof}

\smallskip
\begin{remark}[Single communication round] Algorithms $4$ (FBF) and $5$ (FoRB) are the only single-layer, fixed-step algorithms for GNE seeking in (non-cocoercive, non-striclty) monotone generalized (aggregative) games. The main advantage of Algorithm $5$ is that it requires only one communication round (between the agents and the central coordinator) per iteration instead of the two required by the FBF, see also Remark \ref{rem:2CR}.
{\hfill $\square$}
\end{remark}

\subsection{Customized preconditioned proximal-point algorithm} \label{sec:cPPP}
In this subsection, we focus on a particular class of aggregative games, where the cost functions have the form
\begin{equation} \label{eq:CF-SS}
J_i(x_i,\bs x_{-i}) = g_i(x_i)+ ({C \,\avg(\bs x) )}^\top x_i,
\end{equation}
where $C = C^\top$ is a symmetric matrix.
We emphasize that this particular structure arises in several engineering applications, where $x_i$ denotes the usage level of a certain commodity, whose disutility is modeled by the cost function $g_i(x_i)$, while the term $C \avg(\bs x)$ represents a price function that linearly depends on the average usage level of the population, see
 \cite{chen:li:louie:vucetic:14,deori2018price,kristoffersen2011optimal, atzeni2013demand,estrella2019shrinking,belgioioso2020energy} for some application examples.

\smallskip
The next statement shows that aggregative games with such special structure are generalized potential games \cite[Def. 2.1]{facchinei2011decomposition}.

\smallskip
\begin{lemma} \label{lem:Pot}
Consider monotone aggregative games with agent cost functions as in \eqref{eq:CF-SS} and $C = C^\top$. There exists a continuous function $\phi:\R^{nN} \rightarrow {\R}$ such that $\partial \phi = P$, with $P = \prod_{i =1}^N \partial_{x_i} \, J_i \left( x_i, \,  \bs x_{-i} \right)$.
{\hfill $\square$}
\end{lemma}
\begin{proof}
For aggregative games with linear coupling functions as in \eqref{eq:CF-SS}, the pseudo-subdifferential $P$ in \eqref{eq:PsGr} reads as
\begin{equation} \textstyle \label{eq:PS_cPP}
P = \prod_{i=1}^N \partial_{x_i} g_i +  \frac{1}{N} (I_N + \1_N \1_N^\top) \otimes C.
\end{equation}
Let $\phi(\bs x) := \sum_{i=1}^n g_i(x_i) + \frac{1}{2 } \bs x^\top \left( \frac{1}{ N}(I_N + \1_N \1_N^\top) \otimes C \right) \bs x$, then it is easy to verify that $\partial \phi = P$.
\end{proof}

\smallskip
It follows by Lemma \ref{lem:Pot} that a v-GNE corresponds to a solution to the optimization problem $\argmin \,  \phi(\bs x) \text{ s.t. } \bs x \in \bs{\mc X}$.
However, in many practical setups, a centralized solution to this problem is not viable since it would require a high degree of coordination among selfish agents and also an ``unbearable overload of information exchange'' \cite[\S 3.3]{facchinei2011decomposition}.
Moreover, distributed optimization algorithms, see e.g. \cite{bertsekas1989parallel}, can only deal with feasible sets $\bs{\mc X}$ in \eqref{eq:G} with Cartesian product structure (namely, the case of \textit{non}-generalized games) and cost  functions with a separable form. This motivates us to investigate a customized algorithm for aggregative games with cost functions as in \eqref{eq:CF-SS}, which we summarize in Algorithm~6 and denote as I-cPPP (or cPPP, when $\theta^k \equiv 0$).

\begin{figure}[t]
\hrule
\smallskip
\textit{Algorithm $6$}: Inertial customized PPP (I-cPPP)\smallskip
\hrule
\medskip
\noindent
\textbf{Initialization}: $0 \leq \theta^k \leq \theta^{k+1} \leq \bar \theta < 1/3$
for all $k \geq 0$;
for all $ i \in \mc I$, $x_i^0 \in \R^{n}$, $0<\alpha_i < \textstyle \norm{A_i}+\frac{N-1}{N}\norm{C}$; $\lambda^0 \in \R^m_{\geq 0}$, $0<\beta<( \frac{1}{N} \sum_{i=1}^N \norm{A_i})^{-1}$.

\medskip
\noindent
\textbf{Iterate until convergence}:\\
1. Local: Strategy update, for all $i \in \mc I$:
\begin{align}
\textstyle \nonumber
&y_i^k = \tilde x_i^k - \alpha_i \big(C\,  \avg(\tilde{\bs x}^k)+A_i^\top \tilde{\lambda}^k \big)\\
\nonumber
& x^{k+1}_i = \textstyle
 \underset{z \in \Omega_i}{\argmin} \;
 g_i(z) +\frac{1}{2\alpha_i} \norm{z - y_i^k}^2
 +\frac{1}{N}{\left( C (z  -  \tilde x_i^k) \right)}^\top z \\[.2em]
 \label{eq:inertialXcPPP}
&\tilde{x}_i^{k+1} = x_i^{k+1} + \theta^k(x_i^{k+1}-x_i^{k})\\[.2em]
&d_i^{k+1} = 2A_i x_i^{k+1} - A_i \tilde x_i^k - b_i
 \nonumber
\end{align}

\noindent
2. Central Coordinator: dual variable update
\begin{align} \nonumber
\lambda^{k+1} &= \textstyle
 \proj_{\R^m_{\geq 0}}\left(\tilde \lambda^k
 +\beta \, \avg(\bs d^{k+1}) \right) \hspace*{8.2em}\\
  \tilde{\lambda}^{k+1} &= \lambda^{k+1} + \theta^k(\lambda^{k+1}-\lambda^{k})
  \label{eq:inertialLcPPP}
\end{align}
\hrule
\end{figure}

\smallskip
The next theorem establishes global convergence of Algorithm $6$ to a v-GNE of aggregative games with linear coupling functions as in \eqref{eq:CF-SS}, when the associated PS $P$ is maximally monotone, as postulated next.

\smallskip
\begin{assumption} \label{ass:Pmonot}
The pseudo-subdifferential mapping $P$ in \eqref{eq:PS_cPP} is maximally monotone over $\bs \Omega$.
{\hfill $\square$}
\end{assumption}

\smallskip
We remark that this assumption is less strict than Assumption \ref{ass:monot}, since the monotonicity of the coupled part $F$ is not required. Necessary and sufficient conditions for the (strong) monotonicity of $P$ for this class of aggregative games are discussed in \cite[Cor.\ 1]{belgioioso2017convexity}. For instance, $C \succcurlyeq 0$ is sufficient to guarantee a maximally monotone PS mapping.

\smallskip
\begin{theorem} \label{th:CPPP}
(\textit{Convergence of I-cPPP (Algorithm $6$)})
 Consider the game in \eqref{eq:Game} with cost functions in \eqref{eq:CF-SS}. Let Assumptions \blue{\ref{ass:CFn},} \ref{ass:CCF}, \ref{ass:vGNEex}, \ref{ass:Pmonot} hold. Then,  the sequence $( \col(\bs x^k, \lambda^k))_{k \in \bb N}$ generated by Algorithm $6$, globally converges to some $\col(\bs x^*,\lambda^*) \in \zer(T)$, where $\bs x^*$ is a v-GNE.
{ \hfill $\square$ }
\end{theorem}
\begin{proof}
See Appendix \ref{proof:Th.CPPP}.
\end{proof}

\smallskip
\begin{remark}[v-GAE seeking via I-cPPP]
As for gradient-based methods, to compute a v-GNE via Algorithm~$6$, the agents must know the population size, $N$. However, if an approximate solution, i.e., a v-GAE, is equally desirable, this requirement can be relaxed by removing the correction term $\frac{1}{N}{\left( C (z  -  \tilde x_i^k) \right)}^\top z$ in the local update of each agent $i \in \mc I$:
$x^{k+1}_i \!\!=\! \textstyle
\prox_{\alpha_i g_i + \iota_{\Omega_i}}
 \big(\tilde x_i^k-\alpha_i (C \avg(\tilde{\bs x}^k) \!+\! A_i^\top \tilde \lambda^k)\big)$.
{\hfill $\square$}
\end{remark}

\medskip
\textit{Algorithm $6$ as a fixed-point iteration}: 
In compact form, the dynamics generated by Algorithm $6$ read as the inertial fixed-point iteration
\begin{subequations} \label{eq:CPPP}
\begin{align} 
\tilde{\bs \omega}^k &= \bs \omega^k + \theta^k(\bs \omega^k-\bs \omega^{k-1})\\
\bs \omega^{k+1} &= \mathrm{J}_{\Phi_{\text{C}}^{-1} T}
(\tilde{\bs \omega}^k ) , \label{eq:CPPPb}
\end{align}
\end{subequations}
where $\bs \omega^k = \col(\bs x^k,  \lambda^k)$ is the stacked vector of primal-dual iterates, $\mathrm{J}_{\Phi_{\text{C}}^{-1} T}=(\Id+\Phi_{\text{C}}^{-1} T)^{-1}$ is the generalized resolvent operator of the mapping $T$ in \eqref{eq:T} with preconditioning matrix
\begin{equation} \textstyle \label{eq:Phi_C}
\Phi_{\text{C}} :=
\Phi + \left[
\begin{smallmatrix}
\frac{1}{N} (I_N - \1_N \1_N^\top) \otimes C & \0\\
\0 & \0
\end{smallmatrix} \right],
\end{equation}
with $\Phi$ as in \eqref{eq:Phi}. The iteration in \eqref{eq:CPPP} corresponds to the inertial proximal-point method in \cite{alvarez2001inertial} applied to the mapping $T$ preconditioned with $\Phi_{\text{C}}$.
When $T$ is maximally monotone (which follows by Assumption \ref{ass:Pmonot}) and the step sizes in the main diagonal of $\Phi$ are set such that $\Phi_{\text{C}}\succ 0$, then $J_{\Phi_{\text{C}}^{-1} T}$ is firmly nonexpansive ($\frac{1}{2}-$averaged) w.r.t. the $\Phi_{\text{C}}-$induced norm, i.e., $\|\cdot \|_{\Phi_{\text{C}}}$. Moreover, if the inertial parameter $\theta^k$ is non-decreasing and small enough, then the inertial fixed-point iteration \eqref{eq:CPPP} converges to some $\bs \omega^*:= \col(\bs x^*, \bs \lambda^*) \in \fix(J_{\Phi_{\text{C}}^{-1} T}) = \zer(T)$ \cite[Th.~2.1, Prop.~2.1]{alvarez2001inertial}, where $\bs x^*$ is a v-GNE. We provide the full convergence analysis in Appendix \ref{proof:Th.CPPP}.

\smallskip
\begin{remark}[cPPP is a single-layer algorithm] 
Both the PPP (Algorithm $3$) and our cPPP (Algorithm $6$) rely on the same fixed-point iteration, which is generated by the proximal-point method. However, while the PPP is double-layer, namely, it requires the solution of a sub-game at each iteration, cPPP is single-layer. The idea behind the cPPP is in fact to exploit the special structure of the pseudo-subdifferential $P$ in \eqref{eq:PS_cPP} to customize the preconditioning matrix, $\Phi_\text{C}$, and in turn solve the inner loop of the PPP with a single implicit iteration, namely, the parallel solution of $N$ local, decoupled, strongly convex optimization problems. 
Our cPPP is devised specifically for the subclass of aggregative games presented in this section, thus its applicability is mainly limited to this class of games.
{\hfill $\square$}
\end{remark}

\smallskip
\begin{remark}[Fully-uncoordinated step sizes] 
Unlike all the previously presented gradient-based algorithms, the choice of the local step sizes and inertial parameters in Algorithm 6 is based on local information only\footnote{except for the population size $N$, which is implicitly necessary for computing an exact v-GNE.}. To the best of our knowledge, this is the first and only inertial, fixed-step v-GNE seeking algorithm that enjoys this important property.
{\hfill $\square$}
\end{remark}

\smallskip
\subsubsection*{Over-relaxed cPPP (Algorithm 6B)}
To conclude this section, we present the over-relaxed variant of cPPP, i.e., or-cPPP. This new method is obtained by substituting the inertial steps of primal and dual variables in Algorithm 6, i.e., \eqref{eq:inertialXcPPP} and \eqref{eq:inertialLcPPP}, respectively, with the relaxation steps
\begin{align}
\tilde{x}_i^{k+1} &= \tilde x_i^{k} + \theta^k (x_i^{k+1}-\tilde x_i^{k}),\\
\tilde{\lambda}_i^{k+1} &= \tilde \lambda^{k} + \theta^k(\lambda^{k+1}- \tilde \lambda^{k}),
\end{align}
where the relaxation sequence $\left(\theta^k\right)_{k \in \bb N}$ must be chosen s.t.
\begin{align} \textstyle \label{eq:orParam}
\theta^k \in [0,2] \; \forall k \in \bb N, \quad \sum_{k \in \bb N} \theta^k(2- \theta^k)=\infty.
\end{align}

Similarly to Algorithm 6, or-cPPP can be compactly cast as the following Krasnosel'skii-Mann fixed-point iteration:
\begin{align} 
\tilde{\bs \omega}^{k+1} &= \tilde{\bs \omega}^{k} + \theta^k (\text J_{\Phi_{\text{C}}^{-1} T}(\tilde{\bs \omega}^k ) - \tilde{\bs \omega}^k). \label{eq:CPPP-or}
\end{align}
Thus, its convergence readily follows by \cite[Prop. 5.16]{bauschke2017convex}, since the generalized resolvent $ J_{\Phi_{\text{C}}^{-1}T} $ is $\frac{1}{2}-$averaged w.r.t. the $\Phi_{\text{C}}-$induced norm \cite[Prop. 23.8]{bauschke2017convex}. While there is no interest in doing under-relaxation with $\theta^k$ less than 1, over-relaxation with $\theta^k$ larger than 1 (close to 2) may be beneficial for the convergence speed, as often observed in practice. Interestingly, the choice of the over-relaxation steps $\theta^k$ in \eqref{eq:orParam} is independent from the properties of the mapping $T$.

\subsection{Alternating inertial steps for averaged operators}\label{sec:AI}

In this subsection, we propose an alternating inertial scheme which is applicable to the algorithms in Sections \ref{sec:CREA} and \ref{sec:A-GNE_alg}, and whose updates can be described as a special fixed-point iteration of an averaged operator. An advantage of this scheme is that the generated even subsequence is contractive (Fej\'er monotone) towards a v-GNE. Furthermore, the inertial extrapolation step sizes, $\theta_k$, can freely vary in $[0, 1)$, namely, they do not need to be monotonically non-decreasing. These requirements are less restrictive than those in \cite{mainge2008convergence}, \cite{alvarez2001inertial}.

Next, we first introduce the idea of alternated inertia in operator-theoretic terms, and then apply it to two v-GNE seeking algorithms, the I-pFB (Algorithm 1B) and the I-cPPP (Algorithm $6$).
Let $R$ be an averaged mapping. The alternating inertial Banach--Picard iteration then reads as follows: 
\begin{subequations}
\label{eq:AltInertia}
\begin{align} 
\tilde{\bs \omega}^k &:= 
\begin{cases}
(1+\theta)\bs \omega^k  -\theta \bs \omega^{k-1} , &
\text{ if $k$ odd,}\\
\bs \omega^k,  &\text{ if $k$ even,}
\end{cases}\\[.2em]
\label{eq:AltInertiab}
 \bs \omega^{k+1} &= R(\tilde{\bs \omega}^k).
\end{align}
\end{subequations}
where $\bs\omega^{-1} = \bs \omega^0$ is the initialization.

\smallskip
\begin{lemma} \label{prop:aIs}
Let $R$ be $\eta-$averaged, with $\fix(R) \neq \varnothing$. Then, the even subsequence $(\bs \omega^{2k+2})_{k \in \bb N}$ generated by \eqref{eq:AltInertia}, with $\theta \in \left( 0,\tfrac{1-\eta}{\eta} \right)$, converges to some $ \overline{\bs \omega} \in \fix(R)$.
{\hfill $\square$}
\end{lemma}

\smallskip
\begin{proof}
The odd and even subsequences in \eqref{eq:AltInertiab} read as
\begin{align} \label{eq:AltInertia_In}
\forall k \in \bb N:
\begin{cases}
\bs \omega^{2k+1} = R(\bs \omega^{2k}),\\
\bs \omega^{2k+2} =
R((1+\theta) \bs \omega^{2k+1} - \theta \bs \omega^{2k} )\\
\hspace*{3em}=
 R \circ \left(
(1+\theta)R - \theta\Id
\right) (\bs \omega^{2k}).
\end{cases}
\end{align}
Let us define the mapping $R_\theta := R \circ \left((1+\theta)R - \theta \,\Id \right)
$. The next lemma shows that, for $\theta$ small enough,  $R_\theta $ is averaged and has the same fixed points of $R$.

\smallskip
\begin{lemma} \label{lem:AI}
Let $R$ be $\eta-$averaged, with $\eta \in (0,1)$, and set $\theta \in (0,(1-\eta)/\eta)$.
The following statements hold:
\begin{enumerate}[(i)]
\item $(1+\theta)R - \theta \,\Id$ is $\mu-$averaged, with $\mu=\eta (1+\theta)$,
\item $\fix((1+\theta)R - \theta \,\Id) = \fix(R)$,
\item $R_\theta$ is $\nu-$averaged, with $\nu = \frac{\eta+\mu -2\eta \mu}{1-\eta \mu} \in (0,1)$,
\item $\fix(R_{\theta}) = \fix(R)$. 
\end{enumerate}
\end{lemma}
\begin{proof}
(i) It directly follows from \cite[Prop. 4.40]{bauschke2017convex}.
(ii) $\omega \in \fix((1+\theta)R - \theta \,\Id) \Leftrightarrow (1+\theta)R(\omega) - \theta (\omega) = \omega \Leftrightarrow (1+\theta)R(\omega) = (1+\theta) \omega \Leftrightarrow \omega \in \fix(R)$. (iii) It follows by \cite[Prop. 4.44]{bauschke2017convex}, since $R_\theta$ is the composition of $R$ and $(1+\theta)R - \theta \,\Id$, that are $\eta-$ and $\mu-$ averaged, respectively. (iv) It follows by \cite[Cor.\ 4.51]{bauschke2017convex} that $\fix(R_{\theta}) =  \fix( R\circ ((1+\theta)R - \theta \,\Id )) = \fix(R) \cap \fix((1+\theta)R - \theta \,\Id )) = \fix(R)$.
\end{proof}

In view of \eqref{eq:AltInertia_In} and Lemma \ref{lem:AI}, the even subsequence in \eqref{eq:AltInertiab} can be recast as
\begin{align}
\bs \omega^{2k+2} = R_{\theta} (\bs \omega^{2k}), \quad \forall k \in \bb N,
\end{align}
where $R_{\theta}$ is $\nu-$averaged, with $\nu \in (0,1)$ given by Lemma \ref{lem:AI} (iii). Thus, the convergence of the sequence $(\bs \omega^{k+2})_{k \in \bb N}$ to some $\bar{\bs \omega} \in \fix(R_{\theta}) = \fix(R)$ follows by \cite[Prop. 5.16]{bauschke2017convex}.
\end{proof}

\smallskip
Finally, we propose some explicit rules to choose the alternating inertial extrapolation step sizes for the pFB (Algorithm 1B) and for the cPPP (Algorithm $6$). In fact, in Section \ref{sec:NS}, we observe via numerical simulations that in some cases these alternating-inertial variants outperform the standard-inertial algorithms in terms of convergence speed. Let us then conclude the section with the associated convergence results.

\smallskip
\begin{corollary}
(\textit{Convergence of alternating inertial pFB (aI-pFB)})
Let Assumptions \blue{\ref{ass:CFn},} \ref{ass:CCF}, \ref{ass:vGNEex} and \ref{ass:coco} hold true. Then, the sequence $( \col(\bs x^k, \lambda^k))_{k \in \bb N}$ generated by Algorithm 1B with extrapolation steps set as 
\begin{align} \label{eq:aIssFB}
\theta^k = \begin{cases}
\theta \in \left[0,\frac{2\delta\gamma-1}{2\delta \gamma} \right), & \text{ if $k$, odd}\\
0  & \text{ if $k$ even}\\
\end{cases}
\end{align}
globally converges to some $\col(\bs x^*,\lambda^*) \in \zer(T)\neq \varnothing$, where $\bs x^*$ is a v-GNE.
{\hfill $\square$}
\end{corollary}

\begin{proof}
The pFB algorithm (Algorithm $1$) reads as the fixed-point iteration in \eqref{eq:R}, where the mapping $R_{\text{FB}}$ is $\eta := (\frac{2\delta \gamma}{4\delta \gamma-1})-$averaged w.r.t. the $\Phi-$induced norm. Therefore, the iteration with alternated inertia and extrapolation step sizes 
$0 \leq \theta < \frac{1-\eta}{\eta} = \frac{2\delta \gamma - 1}{2\delta \gamma}$ converges by Lemma \ref{prop:aIs}.
\end{proof}

\smallskip
\begin{corollary}
(\textit{Convergence of alternating inertial cPPP (aI-cPPP)})
Consider the game in \eqref{eq:Game} with cost functions as in  \eqref{eq:CF-SS}. Let Assumptions \blue{\ref{ass:CFn},} \ref{ass:CCF}, \ref{ass:vGNEex} and \ref{ass:Pmonot} hold true. Then, the sequence $( \col(\bs x^k, \lambda^k))_{k \in \bb N}$ generated by Alg. $6$ with extrapolation steps
\begin{align} \label{eq:aIssPPP}
\theta_k = \begin{cases}
\theta \in [0,1), & \text{ if $k$, odd}\\
0,  & \text{ if $k$ even}\\
\end{cases}
\end{align}
globally converges to some $\col(\bs x^*,\lambda^*) \in \zer(T)\neq \varnothing$, where $\bs x^*$ is a v-GNE.
{\hfill $\square$}
\end{corollary}

\begin{proof}
cPPP reads as the fixed-point iteration in \eqref{eq:CPPP} with $\theta^k=0$ for all $k >0$, where the resolvent mapping $J_{\Phi_{\text{C}}^{-1} T}$ is firmly-nonexpansive, i.e.,  $\eta := \tfrac{1}{2}-$averaged, w.r.t. the $\Phi_{\text{C}}-$induced norm. Therefore, the iteration with alternating inertia and extrapolation step sizes $0 \leq \theta < \frac{1-\eta}{\eta} = 1$ converges by Lemma \ref{prop:aIs}.
\end{proof}

\smallskip
\begin{remark}[Convergence rate]
We recall that pFB, cPPP and their alternating-inertial variants can be compactly cast as fixed-point iterations of some averaged operators. It follows by \cite[Th.~1]{davis2016convergence}
that their sequences of fixed-point residuals, i.e., $\| \bs \omega^{k+1} - \bs \omega^{k} \|^2 $, converge with rate $o(1/(k+1))$.
{\hfill $\square$}
\end{remark}

\section{Illustrative application: Charging control of plug-in electric vehicles}
\label{sec:CCPEV}

To study the performance of the proposed algorithms, we formulate a charging coordination problem for a large population of nooncooperative plug-in electric vehicles (PEV) as a generalized aggregative game, as in \cite[\S 6]{paccagnan2018nash}. In subsection \ref{sec:GF}, we introduce the model for the PEV agents, formalize the charging control game and verify that the necessary technical assumptions are satisfied. In subsection \ref{sec:NS} we compare the performance of our algorithm against some standard methods.

\subsection{Game formulation}
\label{sec:GF}
We adopt the same model in \cite[\S 6]{paccagnan2018nash}. Consider the charging coordination problem for a large population of $N \gg 1$ noncooperative PEV over a time horizon made of multiple charging intervals $\{1,2,\ldots,n \}$. The state of vehicle $i$ at time $t$ is denoted by the variable $s_i(t)$. The time evolution of $s_i(t)$ is described by the discrete-time system 
$$s_i(t+1) = s_i(t) + b_i x_i(t), \quad t = 1, \ldots, n,$$ where $x_i(t)$ denotes the charging control input and $b_i$ the charging efficiency. 

\subsubsection*{Constraints} At each time instant $t$, the charging input $x_i(t)$ must be nonnegative and cannot exceed an upper bound $\bar{x}_i(t) \geq 0$. Moreover, the final state of charge must satisfy $s_i(n+1) \geq \eta_i$, where $\eta_i \geq 0$ is the desired state of charge of vehicle $i$. We assume that each PEV agent $i$ decides on its charging strategy $x_i = \col(x_i(1), \ldots,x_i(n)) \in \Omega_i \subset \R^n$, where the set $\Omega_i$ can be expressed as
\begin{multline} \label{eq:PevLocCon}
\Omega_i :=
\left\{
x_i \in \R^n \left| \quad
0 \leq x_i(t) \leq \bar{x}_i(t), \; \forall t = 1,\ldots,n; \right. \right.
\\
\textstyle
\left.
\text{ and }
\sum_{t=1}^n x_i(t) \geq l_i 
\right\},
\end{multline}
where $l_i  = b_i^{-1} (\eta_i - s_i(1))$ and $s_i(1)$ is the state of charge
at the beginning of the time horizon.

Furthermore, for each time instant $t$, the overall power that the grid can deliver to the PEV is denoted by $N K(t)$, thus introducing the following coupling constraints:
\begin{align}\textstyle \label{eq:PevCC}
\frac{1}{N} \sum_{i=1}^N x_i(t) \leq K(t), \quad \text{for all } t =1,\ldots,n,
\end{align}
which in compact form can be cast as $(\1_N^{\top} \otimes I_n) \bs x \leq N K$, with $K = [K(1), \ldots K(n)]^\top$.

\smallskip
\subsubsection*{Cost functions}
The cost function of each PEV represents its electricity bill over the horizon of length $n$ plus a local penalty term $g_i$ (e.g., the battery degradation cost \cite{ma2016efficient}, \cite{ma2015distributed}), i.e, 
\begin{align} \nonumber
J_i(x_i, \bs x_{-i}) &= \sum_{t=1}^n g_{i,t}( x_i(t)) + p_t \left(\frac{
d(t) + \avg(\bs x(t))}{\kappa(t)}
\right)x_i(t) \\
\label{eq:PevCF}
&=: g_i( x_i) + p(\avg(\bs x))^\top x_i,
\end{align}
where $g_i$ is convex and the energy price for each time interval $p_t:\R_{\geq 0} \rightarrow \R_{>0}$ is monotonically increasing, continuously differentiable and depends on the ratio between the total consumption and the total capacity, i.e., $(d(t) + \avg(\bs x(t)))/\kappa(t)$, where $d(t)$ and $\avg(\bs x(t)):= \frac{1}{N}\sum_{i=1}^N x_i(t)$ represent the non-PEV and PEV demand at time $t$ divided by $N$ and $\kappa(t)$ is the total production capacity divided by $N$ as in \cite[eq. (6)]{ma:callaway:hiskens:13}.

\smallskip
\subsubsection*{Aggregative game} Overall, each PEV $i$, given the charging inputs of the other PEV, aims at solving the following optimization problem:
\begin{align}\label{eq:PEVgame}
(\forall i \in \mc I): \; \left\{
\begin{array}{c l}
\underset{x_i \in \, \R^n}{\argmin}&
g_i(x_i)+ p(\avg(\bs x))^\top x_i
\\
\text{ s.t. }   &
 x_i \in \Omega_i,\\[.2em]
&
(\1_N^{\top} \otimes I_n) \bs x \leq N K,
\end{array} 
\right. 
\end{align}

\smallskip
Next, we show that the proposed charging control game in \eqref{eq:PEVgame} does satisfy our technical setup. The local cost functions $J_i$'s in \eqref{eq:PevCF} are convex w.r.t. the local variable $x_i$, the local constraint sets $\Omega_i$'s in \eqref{eq:PevLocCon} are non-empty (for an appropriate choice of the parameters), convex and compact, the coupling constraints in  \eqref{eq:PevCC} are affine and their intersection with the local constraints non-empty (for an appropriate choice of the parameters), namely, the Slater's condition holds true. Hence, Assumption\blue{s \ref{ass:CFn}} and \ref{ass:CCF} are satisfied. In particular, there exist at least one GNE of the game in \eqref{eq:PEVgame}, see Remark \ref{rem:ExGNE}.

\smallskip
The correspondent PG in \eqref{eq:F} and approximate PG in \eqref{eq:aF} read more explicitly as follows:
\begin{align} \textstyle \label{eq:PevWar}
\tilde{F}(\bs x) &= \col \left(\{ p(\avg(\bs x)) \}_{i \in \mc I}\right),\\
F(\bs x) &= \textstyle
\tilde{F}(\bs x) + \frac{1}{N} \col\left( \nabla_{z}p(z)_{|z = \avg(\bs x) }x_i \}_{i \in \mc I}\right).
\label{eq:PevF}
\end{align}
The following lemma shows the properties of these mappings depending on the choice of the price function $p$ in \eqref{eq:PevCF}.

\smallskip
\begin{lemma}[{\cite[Lemma 3]{paccagnan2018nash}}]
\label{lem:NandW}
The following hold:
\begin{enumerate}[(i)] 
\item For all $i \in \mc I$, let $g_i$ be convex and the price function $p$ be monotone, then $\tilde{F}$ in \eqref{eq:PevWar} is maximally monotone;

\item For all $i \in \mc I$, let $g_i$ be convex and the price function $p$ be affine, i.e., $p(\avg(\bs x)) = C \avg(\bs x)  + c $, with $C \in \R^{n\times n}$, and strongly monotone, i.e., $(C+C^\top)/2 \succ 0$, then $F$ in \eqref{eq:PevF} is strongly monotone.
{\hfill $\square$}
\end{enumerate}
\end{lemma}
\begin{proof}
(i) and (ii) follow from \cite[Lemma 3 (i)]{paccagnan2018nash}.
\end{proof}

\smallskip
\subsection{Numerical analysis}
\label{sec:NS}
In our numerical study we consider an heterogeneous population of PEV playing over a time horizon of $n=24$ charging intervals. All the parameters of the game are drawn from uniform distributions and fixed over the course of a simulation. Specifically, for all $i \in \mc I$, we set: the desired final state of charge $l_i$ in \eqref{eq:PevLocCon} according to $l_i \sim (0.5, \, 1.5)$, where $\sim (\tau_1, \tau_2)$ denotes the uniform
distribution over an interval $(\tau_1, \tau_2)$ with $\tau_1 < \tau_2$; for all $t\in \{ 1,\ldots ,n \}$, the upper charging input bound as $\bar x_i(t) \sim (1,\, 5)$, with probability $0.8$, $\bar x_i(t) = 0$ otherwise. For all $t$, the non-PEV demand $d(t)$ is taken as the typical base demand over a summer day in the United States \cite[Figure 1]{ma:callaway:hiskens:13}; $\kappa(t) = 12$ kW, and the upper bound $K(t) = 0.55$ kW is chosen such that the coupling constraints in \eqref{eq:PevCC} are active in the middle of the night.

\smallskip
In the remainder of this section, we study the convergence properties of the proposed algorithms on two different scenarios characterized by a different choice of the price function $p$ and local cost functions $g_i$, in \eqref{eq:PevCF}.

\smallskip
\subsubsection{Monotone price function} Consider the price function 
\begin{align}\textstyle
p_t\left(\avg(\bs x(t))\right) := 0.15 \left(\frac{d(t) +\avg(\bs x(t)))}{\kappa(t)}\right)^{1.5}, \quad \forall t.
\end{align}
as in \cite[\S VII.B]{ma:callaway:hiskens:13} and a local cost function $g_i$ defined as
\begin{align} \textstyle
g_i(x_i) = \pi_i \left( \sum_{t=1}^n x_i(t) \right)^2 + a_i^\top x_i,\quad \forall i \mc I,
\end{align}
where $\pi_i \sim (0.1,0.8)$ and $a_i(t) \sim (0.1,0.4)$, for all $t \in \{1,\ldots,n\}$. Under these choices, it follows from Lemma \ref{lem:NandW} (i) that the approximate PG in \eqref{eq:PevWar} is maximally monotone. Therefore, a v-GAE of the game in \eqref{eq:PEVgame} can be found with the algorithms in Section \ref{sec:AMAG} and the FoRB (Algorithm $5$).
\begin{figure}
\centering
\includegraphics[width=\columnwidth]{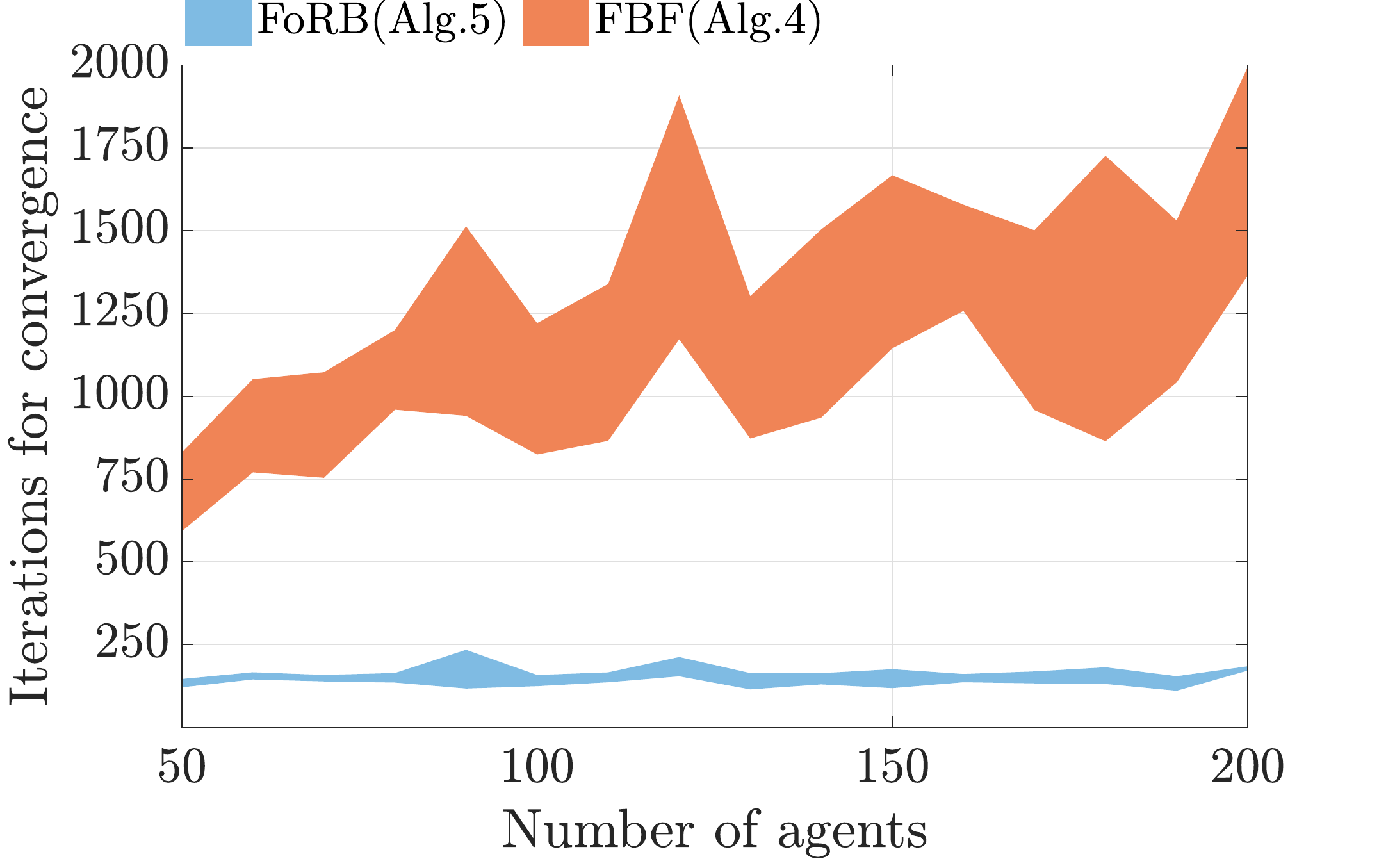}
\caption{
Number of iterations to achieve convergence for FoRB (Alg. 5, $\theta=0$) and FBF (Alg. 4) vs populations size $N$. The areas contain the outcome of 10 random simulations for each $N \in \{ 50, 60, \ldots, 200\}$. Convergence is considered achieved when $\| \bs x^k- \bs x^\star\|/\|\bs x^\star \| \leq 10^{-4}$, where $\bs x^\star$ is a v-GAE.
\label{fig:ALL_MM}
}
\end{figure}

In Fig. \ref{fig:ALL_MM}, we compare the total number of iterations required by FBF (Alg. 4) and FoRB (Alg. 5) to achieve convergence to a v-GAE (i.e., $\|\bs x^k - \bs x^\star\|/ \| \bs x^\star\| \leq 10^{-4}$), over different population sizes $N$ varying from $50$ to $200$ agents. For each $N$, we run 10 simulations with random parameters. On average, FoRB converges at least $5$ times faster than FBF in terms of number of iterations, and, thus, 10 times faster in terms of communication rounds between PEVs and central coordinator. Moreover, unlike FBF, the convergence speed of FoRB seems not affected by increasing the number of agents and randomly varying the parameters of the problem.

\smallskip
\subsubsection{Linear price function}
Consider the price function 
\begin{align}\textstyle
p(\avg(\bs x)) := C \, \avg(\bs x)  + c,
\end{align}
where $C = I_n$, $c = \col(d(1),\ldots,d(n))$, and the local convex cost function $g_i$, for all $i \in \mc I$, as
\begin{align} \textstyle \label{eq:BDC}
g_i(x_i) = \frac{1}{2} x_i^\top Q_i x_i+ p_i^\top x_i.
\end{align}
For instance, the local penalty term $g_i$ in \eqref{eq:BDC} can model a convex quadratic battery degradation cost as in \cite[Eq. (5)]{ma2016efficient}, \cite[Eq. (8)]{ma2015distributed}, possibly plus a quadratic penalty $\|x_i - x_i^{\text{ref}} \|^2$ on the deviation from a preferred charging strategy $x_i^{\text{ref}} \in \Omega_i$.

Under these choices, the pseudo-gradient mapping $F$ in \eqref{eq:PevF} is strongly monotone, by Lemma \ref{lem:NandW} (ii), and Lipschitz continuous, since affine. Thus, it follows by Remark \ref{rem:SMONandLIP_coco} that $F$ is cocoercive.
The unique v-GNE of the game in \eqref{eq:PEVgame} can be found with the algorithms in Section \ref{subsec:pFB} and, since the cost functions have the same structure in \eqref{eq:CF-SS}, with the cPPP.

First, we consider an heterogeneous population of PEV's, by setting the parameters of the local penalty terms $g_i$ in \eqref{eq:BDC} as follows: $Q_i = \diag(q(1),\ldots,q(n))$, $p_i = \col(p_i(1),\ldots,p_i(n))$, with $q_i(t) \sim (0.1,4)$ and $p_i(t) \sim (0.2,2)$, for all $t$.
In Fig. \ref{fig:ALL_het}, we compare the average number of iterations required to achieve convergence (i.e., $\|\bs x^k - \bs x^*\|/ \| \bs x^*\| \leq 10^{-6}$) for pFB (Alg. 1), cPPP (Alg. 6) and their inertial variants, for different population sizes $N$. For each $N$, we run 10 simulations with random parameters and considered the average number of iterations for convergence. The step sizes of all the algorithms are set $1\%$ smaller than their theoretical upper bounds. On average, cPPP outperforms pFB. For both pFB and cPPP, their inertial variants show better performances with respect to the vanilla algorithms. Overall, the over-relaxed cPPP is the fastest among all the considered methodologies. We note that, the convergence speed of all the algorithms seems only mildly affected by the population size.

In Fig. \ref{fig:ALL_homo}, we repeat the same analysis for an homogeneous population of PEV's. Specifically, we set the parameters of the local penalty term in \eqref{eq:BDC} as $Q_i = 0.1$ and $p_i = 0.2$, for all $i \in \mc N$. The performances of all the algorithms improve with respect to the case with heterogeneous agents. On average, cPPP requires less then half the iterations/communication rounds of pFB. For both pFB and CPPP, their inertial variants show better performances with respect to the standard algorithms. Overall, the alternated inertial cPPP (aI-cPPP) and the over-relaxed cPPP (or-cPPP) are the fastest among all the considered methodologies (less than 50 communication rounds with the central coordinator to achieve a precision of $10^{-6}$, independently on the total number of PEVs).

\begin{figure}
\centering
\begin{subfigure}{\linewidth}
\centering\large 
\includegraphics[width=\columnwidth]{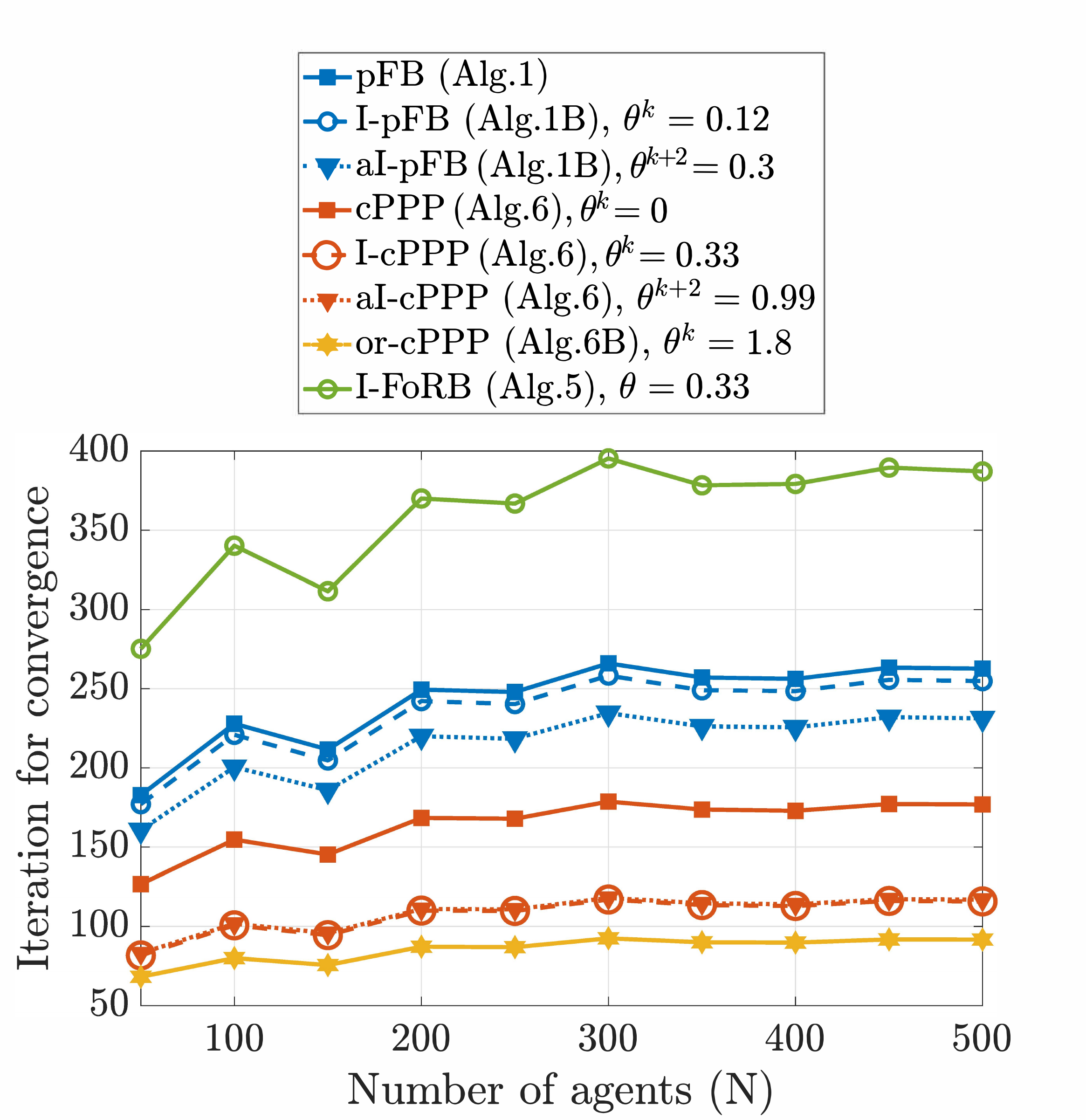}
\caption{
Heterogeneous population of PEVs. The step sizes of pFB and cPPP have been set $1\%$ smaller than their theoretical upper bounds.
\label{fig:ALL_het}
}
\end{subfigure}%
\vspace*{1em}

\begin{subfigure}{\linewidth}
\centering\large
\includegraphics[width=\columnwidth]{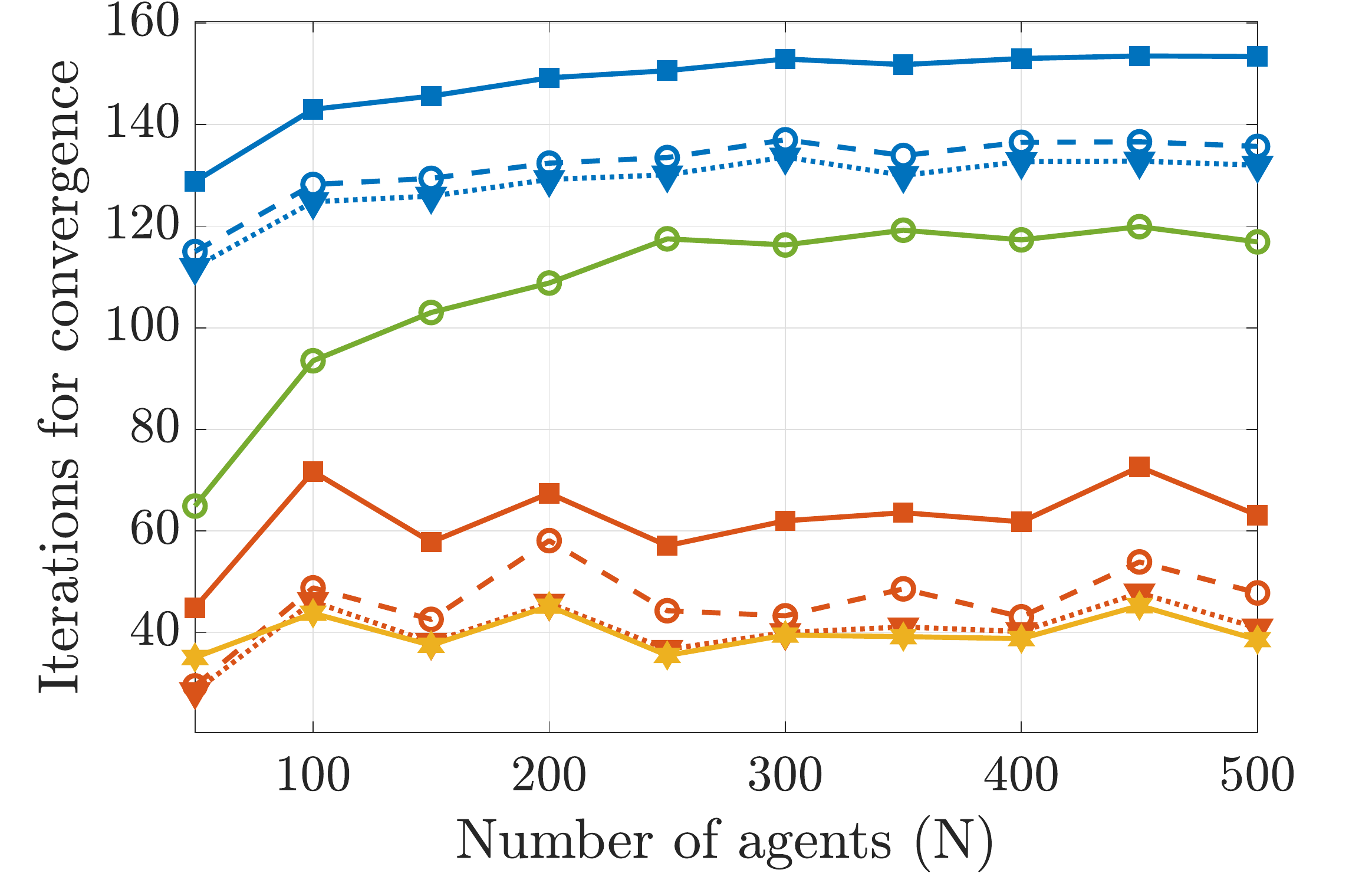}
\caption{
Homogeneous population of PEVs. The step sizes of pFB and cPPP have been set $1\%$ smaller than their theoretical upper bounds.
\label{fig:ALL_homo}
}
\end{subfigure}
\caption{ \small Iterations to achieve convergence for different population sizes. Each polygon is the average over 120 random simulations. Convergence is considered achieved when $\| \bs x^k- \bs x^*\| / \| \bs x^*\| \leq 10^{-6}$.}
\end{figure}

\section{Conclusion and Outlook} \label{sec:Concl}

Generalized Nash equilibrium problems in monotone aggregative games can be efficiently solved via accelerated, semi-decentralized, single-layer, single-communication-per-iteration, fixed-step algorithms. For this class of equilibrium problems, the over-relaxation seems the most effective provably-convergent decentralized way to speed up convergence. The study of adaptive step sizes is left for future work.

\appendix

\subsection{Sufficient conditions for cocoercivity of the pseudo-gradient}
\label{cond:suffCOCO}
In this appendix, we provide some sufficient conditions -- inspired by \cite[\S~III.B]{scutari2014real} -- for cocoercivity of the pseudo-gradient $F$ in \eqref{eq:F} based on properties of the functions $(f_i)_{i \in \mc I}$.

\smallskip
\begin{lemma}
\label{lem:suffCOCO}
For all agent $i\in \mc I$, assume that the function $h_i(\bs x): \bs x \mapsto f_i(x_i,\avg(\bs x))$ is twice continuously differentiable, and let the following conditions hold:
\begin{enumerate}[(i)]
\item $\nabla_{x_i} h_i(\bs x)$ is $\ell_i-$Lipschitz continuous on $\bs \Omega$;

\smallskip
\item there exists a positive constant $\mu_i$ such that
\end{enumerate}
\begin{align*}
\displaystyle \inf_{\bs x \in \bs \Omega} \text{eig}_{\min}(\nabla_{x_i}^2 h_i(\bs x)) \geq \sum_{j \in \mc I \setminus \{i\}} \sup_{\bs x \in \bs \Omega} \|\nabla^2_{x_j, x_i} h_i(\bs x) \| + \mu_i.
\end{align*}
Then, the pseudo-gradient $F(\bs x)=\col\big((\nabla_{x_i} h_i(\bs x))_{i \in \mc I} \big)$ in \eqref{eq:F} is $\gamma-$cocoercive, with $\gamma = \min_{i \in \mc I} \{\mu_i\} / (\max_{i \in \mc I} \{\ell_i\})^2$.
{\hfill $\square$}
\end{lemma}

\smallskip
\begin{proof}
Define the matrix $\Upsilon_F \in \R^{N \times N}$, with entries
\begin{align}
\label{eq:Ups}
[\Upsilon_F]_{i,j} := \begin{cases}
\inf_{\bs x \in \bs \Omega} \text{eig}_{\min}(\nabla_{x_i}^2 h_i(\bs x)), & \text{if } i =j,\\
-\sup_{\bs x \in \bs \Omega} \|\nabla^2_{x_j, x_i} h_i(\bs x) \|, & \text{otherwise}.
\end{cases}
\end{align}
Under the conditions in Lemma~\ref{lem:suffCOCO}~(ii), it follows by the Gershgorin's circle theorem, e.g.  \cite[Th. 2]{feingold1962block}, that $\Upsilon_F $ is positive definite with $\text{eig}_{\min}(\Upsilon_F) = \min_{i \in \mc I} \{\mu_i\}$.
In turn, it follows from \cite[Prop.~5.(c)]{scutari2014real} that the pseudo-gradient $F$ is $\mu-$strongly monotone, with $\mu=\min_{i \in \mc I} \{\mu_i\}$. Moreover, under the conditions in Lemma \ref{lem:suffCOCO}~(i) it follows that $F(\bs x)=\col\big((\nabla_{x_i} h_i(\bs x))_{i \in \mc I} \big)$ is $\ell-$Lipschitz continuous with $\ell= \max_{i \in \mc I} \{\ell_i\}$. Finally, from Remark \ref{rem:SMONandLIP_coco} it follows that $F$ is $\gamma-$cocoercive with $\gamma=\min_{i \in \mc I} \{\mu_i\}/(\max_{i \in \mc I} \{\ell_i\})^2$.
\end{proof}

\subsection{Proof of Theorem \ref{th:FoRB}}
\label{proof:th:FoRB}

To establish global convergence, we show that
\begin{enumerate}[(i)]
\item Algorithm $5$ corresponds to the (preconditioned) inertial FoRB splitting method \cite[Eq. (4.12)]{malitsky2020forward} in \eqref{eq:FP-FoRB};

\item If the step sizes $\{\alpha_i\}_{i \in \mc I}$, $\beta$ and the extrapolation parameter $\theta$ are chosen as in Algorithm $5$, then the assumptions of \cite[Corollary 4.4]{malitsky2020forward} are satisfied, hence $( \col(\bs x^k, \lambda^k))_{k \in \bb N}$ globally converges to some $\col(\bs x^*, \lambda^*) \in \zer(T) \neq \varnothing$, where $\bs x^*$ is a v-GNE.
\end{enumerate}
(i): Let us recast Algorithm $5$ in a more compact form as
\begin{align} 
\nonumber
\bs x^{k+1} &= \diag(\prox_{\alpha_1 g_1 + \iota_{ \Omega_1}}, \ldots, \prox_{\alpha_N g_N + \iota_{ \Omega_N}})\\
\label{eq:xFoRB1}
& \quad  \circ
\big( \tilde{\bs x}^k - \bar \alpha ( 2F(\bs x^{k}) - F(\bs x^{k-1})+ A^\top \lambda^k) \big ),  \\
\label{eq:lamFoRB1}
\lambda^{k+1} &= \proj_{\R^m_{\geq 0}} \big(\tilde \lambda^k + \beta (2A \bs x^{k+1}- A\bs x^k -b) \big),
\end{align}
where $\tilde{\bs x}^k := {\bs x}^k + \theta ({\bs x}^k- {\bs x}^{k-1})$ and $\tilde{ \lambda}^k :=  { \lambda}^k + \theta ({ \lambda}^k- { \lambda}^{k-1})$. 
Since $\diag(\prox_{\alpha_1 g_1 + \iota_{ \Omega_1}}, \ldots, \prox_{\alpha_N g_N + \iota_{ \Omega_N}})=(\Id + \nc_{\bs \Omega} + \bar \alpha G)^{-1}$, it follows by \eqref{eq:xFoRB1} that $(\Id + \nc_{\bs \Omega} + \bar \alpha G)  (\bs x^{k+1}) \in \tilde{\bs x}^k - \bar \alpha ( 2F(\bs x^{k}) - F(\bs x^{k-1})+ A^\top \lambda^k)
$, which leads to
\begin{multline} \label{eq:Alg2-X}
-(2F(\bs x^{k}) - F(\bs x^{k-1})) \in
(\nc_{\bs \Omega} + G)(\bs x^{k+1}) 
+ A^\top \lambda^{k+1} \\
+ \bar \alpha^{-1}(\bs x^{k+1}-\tilde{\bs x}^k) - A^\top(\lambda^{k+1} - \lambda^{k})
\end{multline}
where we used $\bar \alpha^{-1}  \nc_{\bs \Omega}(\bs x^{k+1}) =  \nc_{\bs \Omega} (\bs x^{k+1})$. Equivalently,
it follows from \eqref{eq:lamFoRB1} that $(\Id+\nc_{\R^{m}_{\geq 0}})({ \lambda}^{k+1}) \in \tilde \lambda^k +  \beta \frac{1}{N}(2  A \bs x^{k+1} -  A \bs x^k - b )$, which leads to
\begin{multline} \label{eq:Alg2-L}
\textstyle
-b \in \nc_{\R^{mN}_{\geq 0}}({ \lambda}^{k+1}) - A {\bs x}^{k+1} \\
\textstyle
- A({\bs x}^{k+1}-\bs x^{k}) +  N \beta^{-1}({ \lambda}^{k+1}- \tilde \lambda^{k} ).
\end{multline}
Let $\bs \omega^k := \col(\bs x^k, \lambda^k)$ be the stacked vector of the iterates and $\tilde{\bs \omega}^k = \bs \omega^{k} + \theta(\bs \omega^{k}-\bs \omega^{k-1})$. The inclusions in \eqref{eq:Alg2-X}-\eqref{eq:Alg2-L} can be cast in a more compact form as 
\begin{align} \nonumber
\textstyle
-(2T_1(\bs \omega^k)-T_1(\bs \omega^{k-1})) \in T_2({\bs \omega}^{k+1}) +\Phi ({\bs \omega}^{k+1}- \tilde{\bs \omega}^k),
\end{align}
where $T_1$, $T_2$ and $\Phi$ as in \eqref{eq:T1}, \eqref{eq:T2} and \eqref{eq:SandPhi}, respectively.
By making ${\bs \omega}^{k+1}$ explicit in the last inclusion, we obtain
\begin{multline}  \label{eq:fixPointFoRBproof}
\bs \omega ^{k+1} = (\Id + \Phi^{-1} T_2)^{-1}\\
\circ (\tilde{\bs \omega}^k 
- 2 \Phi^{-1}T_1(\bs \omega ^{k})+ \Phi^{-1} T_1(\bs \omega ^{k-1})),
\end{multline}
which corresponds to \eqref{eq:FP-FoRB}, thus concluding the proof.

\smallskip
(ii): 
Before studying the convergence of iteration \eqref{eq:fixPointFoRBproof}, we show that, if the step sizes are chosen as in Algorithm $5$, then the preconditioning matrix $\Phi$ is positive definite.

\begin{lemma} \label{lem:PMphiLip}
Let $\{\alpha_i\}_{i \in \mc I}$ and $\beta$ be set as in Algorithm $5$. Then,  the following statements hold:
\begin{enumerate}[(i)]
\item $\Phi - \delta I \succeq 0$;

\item $\| \Phi^{-1} \| \leq \delta^{-1}$.
{\hfill $\square$}
\end{enumerate}
\end{lemma}
\begin{proof}
(i): By the generalized Gershgorin circle theorem \cite[Th. 2]{feingold1962block}, each eigenvalue $\mu$ of the matrix $\Phi$ in \eqref{eq:SandPhi} satisfies at least one of the following inequalities:
\begin{align} \label{eq:GGCT1}
\mu & \geq \alpha_i^{-1} - \| A_{i}^\top\| , \quad  &\forall i \in \mc I,\\
\mu & \geq \textstyle N\beta^{-1} -  \sum_{j = 1}^N \| A_{j}^\top\|. \quad  &
\label{eq:GGCT3}
\end{align}
Hence, if we set the step sizes $\{ \alpha_i \}_{i \in \mc I}, \beta$ as in Algorithm $5$, the inequalities \eqref{eq:GGCT1}-\eqref{eq:GGCT3} yield to $\mu \geq \delta $, where $\delta>0 $ by design choice. It follows that the smallest eigenvalue of $\Phi$, i.e., $\text{eig}_{\min}(\Phi)$, satisfies $\text{eig}_{\min}(\Phi) \geq \delta >0$. Hence, $\Phi-\delta I \succeq 0$.\\
(ii): Let $\text{eig}_{\max}(\Phi)$ be the largest eigenvalue of $\Phi$. We have that $\text{eig}_{\max}(\Phi)\ge \text{eig}_{\min}(\Phi) \geq \delta $. Moreover, $\| \Phi \|=\text{eig}_{\max}(\Phi) \ge \text{eig}_{\min}(\Phi) = \frac{1}{\| \Phi^{-1} \|} \geq \delta$. Hence $\| \Phi^{-1} \| \leq \delta^{-1}$. 
\end{proof}

\smallskip
Since $\Phi^{-1}$ is $\delta^{-1}-$Lipschitz, by Lemma \ref{lem:PMphiLip} (ii), and $T_1$ is $\ell-$Lipschitz, by Assumption \ref{ass:monot}, then their composition, i.e., $\Phi^{-1} \circ T_1$, is $\tau-$Lipschitz continuous, with $\tau:=\delta^{-1} \ell < (1-3\theta)/2$, since $\delta > 2 \ell/(1-3\theta)$, by design choice.

The fixed-point iteration \eqref{eq:fixPointFoRBproof}, that corresponds to Algorithm $5$ by the first part of this proof, is the inertial FoRB splitting algorithm on the mappings $\Phi^{-1}T_1$ and $\Phi^{-1}T_2$.
The convergence of \eqref{eq:fixPointFoRBproof} to some $\bs \omega^*:=\col(\bs x^*,\bs \lambda^*) \in \zer(T_1+T_2)$ follows by \cite[Corollary 4.4, Remark 2.7]{malitsky2020forward}, since
$\Phi^{-1}T_1$ and $\Phi^{-1}T_2$ are maximally monotone in the $\Phi-$induced norm and $\Phi^{-1}T_1$ is $\tau-$Lipschitz continuous, with $\tau< (1-3\theta)/2$. To conclude, we note that $\bs \omega^* \in \zer(\Phi^{-1}T_1+ \Phi^{-1} T_2)= \zer(T)$, since $\Phi\succ 0$, by Lemma \ref{lem:PMphiLip} (i), and $T_1+T_2=T$. Since the limit point $\bs \omega^* \in \zer(T) \neq \varnothing$, then $\bs x^*$ is a v-GNE of the game in \eqref{eq:Game}, by Proposition \ref{pr:UvGNE}, thus concluding the proof.
{\hfill $\blacksquare$}

\subsection{Proof of Theorem \ref{th:CPPP}}
\label{proof:Th.CPPP}
To establish global convergence, we show that
\begin{enumerate}[(i)]
\item Algorithm $6$ corresponds to the inertial proximal-point method \cite[Th.~2.1]{alvarez2001inertial} in \eqref{eq:CPPP};

\item If the step sizes $\{\alpha_i\}_{i \in \mc I}$, $\beta$ and the inertial parameters $\theta^k$ are chosen as in Algorithm $6$, then the assumptions of \cite[Th.~2.1, Prop.~2.1]{alvarez2001inertial} are satisfied, hence $( \col(\bs x^k, \lambda^k))_{k \in \bb N}$ globally convergences to some $\col(\bs x^*, \lambda^*) \in \zer(T)$, where $\bs x^*$ is a v-GNE.
\end{enumerate}

\smallskip
(i): With some cosmetic manipulations, we can rewrite the local primal update of agent $i$ as the solution to
\begin{align*} 
x_i^{k+1} = \textstyle \underset{z \in {\Omega}_i}{\argmin} \, 
J_i \left( z, \tilde x_{-i}^k \right)
+{(A_i^\top \tilde \lambda_i^k)}^\top z +  \frac{1}{2\alpha_i} \norm{z-\tilde x_i^k}^2,
\end{align*}
with $J_i$ as in \eqref{eq:CF-SS}.
Equivalently, $x_i^{k+1}$ must satisfy 
\begin{multline*} \textstyle
\0_n \in \partial_{x_i}\big(
J_i \left( x^{k+1}_i, \tilde x_{-i}^k \right)
+{(A_i^\top \tilde \lambda_i^k)}^\top x^{k+1}_i\\
\textstyle
 +  \frac{1}{2\alpha_i} \norm{x^{k+1}_i -\tilde x_i^k}^2
\big).
\end{multline*}
Since $\partial_{x_i}
J_i ( x^{k+1}_i, \tilde x_{-i}^k ) = \partial_{x_i} g_i(x_i^{k+1}) + \frac{1}{N} \sum_{j \neq i}^N  C\tilde  x_j^k+ \frac{2}{N} C x_i^{k+1}$, then the previous inclusion can be rewritten as
\begin{multline} \textstyle \label{eq:inclCIPPP}
\0_n \in \partial_{x_i} g_i(x_i^{k+1}) + \frac{2}{N} C x_i^{k+1}
 +\frac{1}{N} [(\1_N^\top \otimes C){\bs x}^{k+1} -C x_i^{k+1}] \\
\textstyle
+ A_i^\top  \lambda_i^{k+1}
 +  \frac{1}{\alpha_i} (x^{k+1}_i -\tilde x_i^k) - \frac{1}{N} [(\1_N^\top \otimes C)({\bs x}^{k+1}-\tilde{\bs x}^k) \\
 \textstyle
 +  C ( x_i^{k+1}-\tilde x_i^k)]  - A^\top_i(\lambda^{k+1}- \lambda^k).
\end{multline}
By stacking-up the inclusions \eqref{eq:inclCIPPP}, for all $i \in \mc I$, we obtain
\begin{multline}\label{eq:bla}
\textstyle
\0_{nN} \in G(\bs x^{k+1}) + \frac{2}{N}(I_N \otimes C) \bs x^{k+1}  -\\
\textstyle
 \frac{1}{N} ((I_N - \1 \1^\top )\otimes C) \bs x^{k+1} + A^\top \bs \lambda^{k+1} 
+ \bar \alpha^{-1}(\bs x^{k+1} - \tilde{\bs x}^k)\\
\textstyle
 + \frac{1}{N} ((I_N - \1 \1^\top )\otimes C) (\bs x^{k+1}- \tilde{\bs x}^{k}) - A^\top (\lambda^{k+1}-\tilde  \lambda^k),
\end{multline}
where the first 3 terms on the right-hand side correspond to the pseudo-subdifferential mapping $P(\bs x^{k+1})$, i.e., $G(\bs x^{k+1}) + \frac{2}{N}(I_N \otimes C) (\bs x^{k+1})  - \frac{1}{N} ((I_N - \1 \1^\top )\otimes C) (\bs x^{k+1})= P(\bs x^{k+1})$. 

It follows by the dual update in Algorithm $6$ that $(\Id+\nc_{\R^{m}_{\geq 0}})({ \lambda}^{k+1}) \in \lambda^k + \beta \frac{1}{N}(2  A \bs x^{k+1} -  A \tilde{\bs x}^k - b )$, which yields
\begin{multline} \label{eq:Alg7_dual}
\textstyle
\0_m \in \nc_{\R^{mN}_{\geq 0}}({\lambda}^{k+1}) - (A {\bs x}^{k+1}-b) \\
\textstyle
- A({\bs x}^{k+1}-\tilde{\bs x}^{k}) + N \beta^{-1}({ \lambda}^{k+1}-\tilde{ \lambda}^{k} ).
\end{multline}

Let $\bs \omega^k := \col(\bs x^k, \lambda^k)$ be the stacked vector of the iterates. The inclusions in \eqref{eq:bla}-\eqref{eq:Alg7_dual} can be cast in a compact form as 
\begin{align} \nonumber
\textstyle
\0 \in T({\bs \omega}^{k+1}) +\Phi_{\text C} ({\bs \omega}^{k+1}- \tilde{\bs \omega}^k),
\end{align}
where $T$ and $\Phi_\text{C}$ as in \eqref{eq:T} and \eqref{eq:Phi_C}, respectively. 
By making ${\bs \omega}^{k+1}$ explicit in the last inclusion, we obtain
\begin{equation}  \label{eq:fixPointCIPPPproof}
\bs \omega ^{k+1} = (\Id + \Phi_{\text C}^{-1} T)^{-1}(\tilde{\bs \omega}^k),
\end{equation}
where the auxiliary updates can be cast in a compact form as
\begin{equation} \label{eq:fixPointCIPPPproof_2}
\tilde{\bs \omega}^k = \bs \omega^{k} + \theta^k (\bs \omega^{k}-\bs \omega^{k-1}).
\end{equation}
By combining \eqref{eq:fixPointCIPPPproof} and \eqref{eq:fixPointCIPPPproof_2}, we obtain the fixed-point iteration in \eqref{eq:CPPP},
that corresponds to the iteration in \cite[Th.~2.1]{alvarez2001inertial} applied on $\Phi_{\text C}^{-1}T$ and, thus, concludes the proof.

\smallskip
(ii): The following Lemma shows that, if the step sizes are chosen as in Algorithm $6$, then the preconditioning matrix $\Phi_{\text C}$ is positive definite.

\begin{lemma} \label{lem:posDefCIPPP}
Let $\{\alpha_i\}_{i \in \mc I}$ and $\beta$ be set as in Algorithm $6$, then, $\Phi_{\text{C}}  \succ 0$.
{\hfill $\square$}
\end{lemma}
\begin{proof}
This proof follows the same technical reasoning of the proof of Lemma \ref{lem:PMphiLip} (i) and, thus, is omitted.
\end{proof}

The fixed-point iteration \eqref{eq:CPPP}, that corresponds to Algorithm $6$ by the first part of this proof, is the inertial proximal-point algorithm applied on the operator $\Phi_{\text C}^{-1}T$.
The convergence of \eqref{eq:CPPP} to some $\bs \omega^*=\col(\bs x^*,\bs \lambda^*) \in \zer(\Phi_{\text C}^{-1} T)$ follows by \cite[Th.~2.1, Prop.~2.1]{alvarez2001inertial}, since
$\Phi_{\text C}^{-1}T$ is maximally monotone in the $\Phi_{\text C}-$induced norm and $0 \leq \theta^k \leq \theta^{k+1} \leq \bar \theta < 1/3$, for all $k>0$. To conclude, we note that $\bs \omega^* \in \zer(\Phi_{\text C}^{-1}T)= \zer(T)$, since $\Phi_{\text C} \succ 0$, by Lemma \ref{lem:posDefCIPPP}. Since the limit point $\bs \omega^* \in \zer(T) \neq \varnothing$, then $\bs x^*$ is a v-GNE of the game in \eqref{eq:Game}, by Proposition \ref{pr:UvGNE}, concluding the proof.
{\hfill $\blacksquare$}


\balance
\bibliographystyle{IEEEtran}
\bibliography{IEEEfull,biblio,library}


\end{document}